\documentclass[12pt, a4paper]{amsart}
\usepackage{amsfonts,amsmath, amsthm, amssymb}
\usepackage{latexsym}
\usepackage{enumitem}

\usepackage[dvips]{graphicx}
\usepackage{psfrag}
\usepackage{xcolor}
\usepackage{tikz}
\usetikzlibrary{shapes,arrows,positioning,intersections}
\usepackage{hyperref}

\newcommand{\cal}{\mathcal}
\newtheorem{theorem}{Theorem}[section]
\newtheorem{lemma}[theorem]{Lemma}
\newtheorem{corollary}[theorem]{Corollary}
\newtheorem{proposition}[theorem]{Proposition}

\newtheorem{definition}[theorem]{Definition}
\newtheorem{remark}[theorem]{Remark}
\numberwithin{equation}{section}

\newcommand{\useonX}{} 
\newcommand{\useonM}{\underline}

\newcommand{\tm}{\useonX{m}}
\newcommand{\tB}{\useonX{B}}
\newcommand{\tS}{\useonX{S}}
\newcommand{\tc}{\useonX{c}}
\newcommand{\tv}{\useonX{v}}
\newcommand{\tw}{\useonX{w}}
\newcommand{\hB}{\useonM{B}}
\newcommand{\hS}{\useonM{S}}
\newcommand{\hc}{\useonM{c}}
\newcommand{\hv}{\useonM{v}}
\newcommand{\hw}{\useonM{w}}
\newcommand{\hm}{\useonM{m}}

\newcommand{\Pa}{\mathbf{P}}
\newcommand{\Fu}{\mathbf{F}}

\newcommand{\ideal}{\partial X}
\newcommand{\sqbd}{\partial^2 X}
\newcommand {\N}{\mathbb{N}} 
\newcommand {\Z}{\mathbb{Z}} 
\newcommand {\R}{\mathbb{R}} 

\DeclareMathOperator{\vol}{vol}
\DeclareMathOperator{\diam}{diam}

\DeclareMathOperator{\id}{Id}
\DeclareMathOperator{\supp}{supp}
\DeclareMathOperator{\pr}{pr}

\DeclareMathOperator{\Leb}{Leb}

\DeclareMathOperator{\grad}{grad}

\DeclareMathOperator{\inj}{inj}

\makeatletter
\@namedef{subjclassname@2020}{\textup{2020} Mathematics Subject Classification}
\makeatother

\begin{document}
\title[Closed geodesics]{Closed geodesics on surfaces without conjugate points}
\author{Vaughn Climenhaga}
\author{Gerhard Knieper}
\author{Khadim War}
\address{Dept.\ of Mathematics, University of Houston, Houston, TX 77204}
\address{Dept.\ of Mathematics, Ruhr University Bochum, 44780 Bochum, Germany}
\address{IMPA, Estrada Dona Castorina 110, Rio de Janeiro, Brazil
}
\email{climenha@math.uh.edu}
\email{gerhard.knieper@rub.de}
\email{khadim@impa.br}

\date{\today}


\thanks{V.C.\ is partially supported by NSF grant DMS-1554794.}
\thanks{G.K. and K.W. are partially supported by the German Research Foundation (DFG),
CRC/TRR 191, \textit{Symplectic structures in geometry, algebra and dynamics.}}

\maketitle

\begin{abstract}
We obtain Margulis-type asymptotic estimates for the number of free homotopy classes of closed geodesics on certain manifolds without conjugate points.
Our results cover all compact surfaces of genus at least 2 without conjugate points.  
\end{abstract}

\section{Introduction}

\subsection{Main results}

Given a closed Riemannian manifold $(M,g)$,
it is well known that each free homotopy class of loops contains at least one closed geodesic.
Let $P(t)$ denote the set of free homotopy classes containing a closed geodesic with length at most $t$, and $\#P(t)$ its cardinality.  From the point of view of dynamics and geometry  it is of considerable interest 
 to estimate $\#P(t)$ as $t$ tends to infinity. 

\begin{theorem}\label{thm:surfaces}
Let \((M,g)\) be a closed connected surface of genus at least $2$ without conjugate points. Then
\begin{equation}\label{eqn:margulis}
\# P(t) \sim \frac{e^{ht}}{ht},
\end{equation}
where $h$ is the topological entropy of the geodesic flow on the unit tangent bundle $SM$, and the notation $f_1\sim f_2$ means
$\frac{f_1(t)}{f_2(t)} \to 1$ as $t\to\infty$.
\end{theorem}

In \S\ref{sec:history} we discuss the history of the Margulis asymptotics \eqref{eqn:margulis} and various related results, and in \S\ref{sec:outline} we outline the proof, which uses ideas from the original work of Margulis \cite{Mar1,Mar2} and from a recent preprint of Ricks \cite{Ri}. 
The following result clarifies the ingredients needed for our argument; precise definitions are in \S\ref{sec:uniq}. 

\begin{theorem}\label{thm:main}
Let $(M,g)$ be a closed Riemannian manifold without conjugate points that admits a background metric $g_0$ of negative curvature and satisfies the divergence property (intersecting geodesics in the universal cover diverge).  If the geodesic flow has a unique measure of maximal entropy and if this measure is almost expansive in the sense of Definition \ref{def:expansive}, then \eqref{eqn:margulis} holds.
\end{theorem}

We proved in \cite{CKW} that the hypotheses of Theorem \ref{thm:main} are satisfies for every surface of genus at least 2 without conjugate points, so this result implies Theorem \ref{thm:surfaces}.
By results of Manning \cite{Man}, a background metric of negative curvature forces the geodesic flow of any other metric to have positive topological entropy, so $h>0$ in \eqref{eqn:margulis}.

In the course of the proof, we establish an equidistribution result for periodic orbits with lengths in $(t-\epsilon,t]$, which we believe to be of independent interest. The corresponding result for orbits with lengths in $(0,t]$ was proved in \cite{CKW}.

\begin{theorem}\label{thm:equidist}
Let $(M,g)$ be as in Theorem \ref{thm:main}, and fix any $\epsilon>0$.  Given $t>0$, let $C(t)$ be any maximal set of pairwise non-free-homotopic closed geodesics with lengths in $(t-\epsilon,t]$, and consider the measure
\begin{equation}\label{eqn:nut}
\nu_t = \frac 1{\# C(t)} \sum_{c\in C(t)} \frac{\Leb_c}{t},
\end{equation}
where $\Leb_c$ is Lebesgue measure (length) along the curve $\dot{c}$ in the unit tangent bundle $SM$. Then as $t\to\infty$, the measures $\nu_t$ converge in the weak* topology to the measure of maximal entropy.
\end{theorem}

The results in \cite{CKW}, and our arguments here, apply to a class of higher-dimensional manifolds as well; the following is a consequence of Theorems \ref{thm:main} and \ref{thm:equidist} together with \cite{CKW}.

\begin{theorem}\label{thm:extend-CKW}
Let $(M,g)$ be a closed Riemannian manifold without conjugate points that admits a background metric $g_0$ of negative curvature and satisfies the divergence property.  Suppose that the fundamental group $\pi_1(M)$ is residually finite and 
that the entropy gap condition in \S\ref{sec:geom-hyp} is satisfied.  Then \eqref{eqn:margulis} holds, as does the equidistribution result in Theorem \ref{thm:equidist}.
\end{theorem}

\begin{remark}\label{rmk:problem}
To illustrate the difficulties involved in proving the above results, let us first remark that the original proof of \eqref{eqn:margulis} by Margulis uses the local product structure of the continuous stable and unstable foliations of an Anosov flow to construct a ``flow box''. In our setting, where the flow need not be Anosov, analogous stable and unstable bundles were constructed using Jacobi fields by Hopf \cite{Ho54} for surfaces and by Green \cite{Gr58} in any dimension, but the question of their product structure is more subtle. 
For surfaces (and more generally rank 1 manifolds) of non-positive curvature, the bundles are at least transverse in some open region; however, in our ``no conjugate points'' setting, even this much is not known. This transversality would follow if the bundles were continuous \cite{gK86}, as they are in non-positive curvature,
but an example provided by Ballmann, Brin, and Burns \cite{BBB} shows that continuity may fail under the assumption of no conjugate points. In particular, the natural construction of Margulis does not automatically produce any flow boxes with which to carry out the proof.
See \S\ref{sec:beyond-negative} and \S\ref{sec:outline} for an explanation of how this difficulty is overcome.
\end{remark}

\subsection{History}\label{sec:history}

Here we recall some of the history of \eqref{eqn:margulis} and weaker estimates, referring to the survey by Sharp \cite{Mar2} for a more detailed overview and for a discussion of the stronger asymptotic estimates that are available in certain settings.

\subsubsection{Types of estimates}

We are interested in manifolds for which $\#P(t)$ grows exponentially quickly. In this case one can identify three levels of asymptotic results, each stronger than the last.
\begin{itemize}
\item \emph{Exponential growth rate:} existence of the limit
\begin{equation}\label{eqn:h}
h = \lim_{t\to\infty} \frac 1t \log \#P(t), 
\end{equation}
and identification of $h$ with a dynamical or geometric quantity such as topological entropy or the growth rate of balls in the universal cover.
\item \emph{Uniform counting bounds:} existence of $A,B>0$ such that
\begin{equation}\label{eqn:unif}
\frac At e^{ht} \leq \#P(t) \leq \frac Bt e^{ht} \text{ for sufficiently large }t.
\end{equation}
\item \emph{Multiplicative asymptotics:} the result in \eqref{eqn:margulis}.
\end{itemize}
Broadly speaking, there are two types of approaches to proving these estimates. One approach relies on studying the dynamical properties of the geodesic flow on $SM$, with an important role played by the measure of maximal entropy. The other approach relies on methods from geometry, analysis, and geometric group theory, and leans heavily on the fundamental group $\pi_1(M)$, its isometric action on the universal cover of $M$ as the group $\Gamma$ of deck transformations,
and the correspondence between free homotopy classes of curves on $M$ and conjugacy classes in $\Gamma \cong \pi_1(M)$.\footnote{This correspondence means that one can obtain estimates on $\#P(t)$ by counting conjugacy classes. It is worth pointing out that the question of growth of conjugacy classes is distinct from the question of growth of group elements; there are groups of exponential growth with only two conjugacy classes \cite[Corollary 1.2]{dO10}.}
As will become clear later, our proof combines elements of both approaches.

\subsubsection{Negative curvature and uniform hyperbolicity}

When all sectional curvatures of $(M,g)$ are negative, each free homotopy class is represented by a unique closed geodesic, so $\#P(t)$ represents the number of closed geodesics with length at most $t$. 
In the 1950's Huber \cite[Satz 10]{Hu} proved for closed surfaces of constant curvature $-1$ that \eqref{eqn:margulis} holds with $h=1$. He used a relation between the spectrum of the Laplacian and the length spectrum (the set of lengths of all closed geodesics) 
which in its general form is given by Selberg's trace formula \cite{aS56} and involves a sum over conjugacy classes in $\Gamma$. 
Proofs using Selberg's methods can be found in \cite{hM72,dH76,pB92}.\footnote{In this context \eqref{eqn:margulis} is often called the prime number theorem for compact Riemann surfaces: note that the classical prime number theorem says that if $\pi(x)$ is the number of primes less than or equal to $x$, then $\pi(e^t) \sim \frac{e^t}t$.}
This formula extends to locally symmetric spaces of negative curvature in arbitrary dimension, where it was later used to deduce \eqref{eqn:margulis} by Gangolli \cite{rG77}.

To deal with variable curvature, where there is no Selberg trace formula, different arguments are needed. In 1966, Sinai \cite{Si} used dynamical methods to prove that if $(M,g)$ has dimension $n$ and all of its sectional curvatures lie between $-b^2$ and $-a^2$ for some $b\geq a>0$, then the lower and upper limits in \eqref{eqn:h} lie between $(n-1)a$ and $(n-1)b$;
however, he did not prove existence of the limit.

A striking improvement of Sinai's bounds was given by Margulis in his 1969 thesis, which established \eqref{eqn:margulis} -- and hence \eqref{eqn:h} and \eqref{eqn:unif} as well -- for every $(M,g)$ with (variable) negative sectional curvatures. Margulis's arguments are dynamical and apply to any topologically mixing Anosov flow, where in this dynamical setting $P(t)$ denotes the set of periodic orbits with period at most $t$.\footnote{This yields the geometric conclusion because the geodesic flow on $SM$ is Anosov when $M$ has negative curvature.} 

At the heart of Margulis's proof is the construction of an invariant mixing measure $m$ with the following properties:
\begin{enumerate}
\item the conditional measures of $m$ on stable and unstable manifolds contract and expand with a uniform rate under the flow;
\item $m$ is a measure of maximal entropy for the flow;
\item $m$ is the limiting distribution of periodic orbits in the sense of Theorem \ref{thm:equidist}.
\end{enumerate}
Margulis's results were announced without proof in \cite{Mar1}, in which the right-hand side of \eqref{eqn:margulis} had the form $c t^{-1} e^{ht}$ for some unspecified constant $c$, which was stated to be equal to the entropy $h$ in the constant curvature case. The stronger result that $c=h$ in general was contained in his thesis itself, which was not published until 2004 \cite{Mar2}. In the meantime, 
this equality was established in the 1984 thesis of Toll \cite{cT84}, which gave an account of Margulis's ideas in English (albeit an unpublished one that only circulated informally); his description of Margulis's approach was also reproduced in the 1995 book of Katok and Hasselblatt \cite[\S20.5--6]{KH}.

After Margulis's work was announced but before his arguments were widely known, several important alternate dynamical approaches were developed. First came the work of Bowen in the early 1970's, which used the specification and expansivity properties satisfied by Anosov and Axiom A flows to prove \eqref{eqn:h} and \eqref{eqn:unif} in \cite{rB72-hyp} and \cite[Theorem 4.1]{rB72-geod}, respectively, and to show that periodic orbits are equidistributed with respect to an ergodic measure of maximal entropy.  In \cite{Bo} Bowen proved that this measure of maximal entropy is in fact unique, and it is now referred to as the Bowen--Margulis measure. In 1983, Parry and Pollicott \cite{PaPo} gave an alternative proof of \eqref{eqn:margulis} based on symbolic dynamics and zeta functions, which covers Axiom A flows as well.

In 1990, Kaimanovich \cite{Kai} showed that on a manifold with negative curvature, the Bowen--Margulis measure can also be obtained by building a family of measures on the boundary at infinity using a construction due to Patterson \cite{Pat} and Sullivan \cite{Su}.
These ideas were used by Roblin in \cite{tR} to give another proof of \eqref{eqn:margulis}, which 
even works for certain non-compact and negatively curved Alexandrov metric spaces was provided by Roblin in \cite{tR}. See \cite{PPS15} for a comprehensive treatment of this approach.


\subsubsection{Beyond negative curvature}\label{sec:beyond-negative}

Outside of the uniformly hyperbolic setting, there are fewer comprehensive results available. In 1982 Katok used a closing lemma in Pesin theory to show that for a smooth flow with topological entropy $h>0$ and no fixed points, the exponential growth rate of the number of periodic orbits is at least $h$ \cite[Theorem 4.1]{Kat}, yielding one inequality in \eqref{eqn:h}. Recently, Lima and Sarig strengthened this to the lower uniform bound in \eqref{eqn:unif} \cite[Theorem 8.1]{LS19} when the flow is on a manifold of dimension $3$ (for geodesic flows, this requires $M$ to be a surface so that $\dim SM = 3$). 

On the other hand, without the assumption of uniform hyperbolicity, it is easy to produce flows with uncountably many periodic orbits below a given length, which suggests that we should really be counting equivalence classes of periodic orbits, rather than individual orbits. In the most general setting it is not clear what equivalence relation would be appropriate, but for geodesic flow on $SM$ it is natural to consider free homotopy of the underlying geodesics on $M$, and from now on we restrict our attention to this setting. 

Now that $\#P(t)$ counts free homotopy classes of closed curves, the upper bound in \eqref{eqn:h} holds without any further assumptions on $M$ \cite[Satz 2.1]{K1}: growth of homotopy classes creates entropy.  However, the lower bound in \eqref{eqn:h} no longer holds in general when we are counting homotopy classes instead of individual geodesics, as evidenced by the fact that the sphere admits metrics for which the geodesic flow has positive topological entropy \cite{vD88,BG89}. Thus some assumptions must be placed on $M$ to proceed further.

In 1983, the second author studied closed Riemannian manifolds with nonpositive curvature which admit a geodesic that does not bound a flat half plane \cite{K1} (in particular, this includes manifolds of geometric rank $1$) and proved the exponential growth rate \eqref{eqn:h}, where $h$ is the topological entropy. This used geometric methods, and in particular the result of Manning \cite{Man} that $h$ gives the growth rate of the volume of balls in the universal cover.

In the rank $1$ nonpositive curvature setting, the second author \cite{K4,K2} adapted the ``boundary at infinity'' approach to obtain the uniform bounds \eqref{eqn:unif}.\footnote{The upper bound in \cite{K4} does not have the $t$ in the denominator, but \cite[Proposition 5.3.6]{K2} explains how to obtain it.}
In \cite{K5}, he showed that the measure of maximal entropy constructed from the Patterson--Sullivan measure is unique, and Babillot used the product structure of this measure to prove the mixing property \cite{Bab}.
We remark that in this approach, this product structure is described in terms of the boundary at infinity rather than the stable and unstable bundles, an idea which also plays a central role in resolving the difficulties described in Remark \ref{rmk:problem}.

\begin{remark}
An alternate proof of uniqueness in rank $1$ based on a variant of Bowen's specification property was recently given by the first author together with Burns, Fisher, and Thompson \cite{BCFT}.  As with Bowen's original work, this approach can be used to deduce the uniform estimates \eqref{eqn:unif}, although this is not done explicitly in the paper,\footnote{See \cite[Corollary 4.8 and Proposition 6.4]{BCFT} for results toward \eqref{eqn:unif}.} and these techniques do not lead to the stronger Margulis estimates \eqref{eqn:margulis}.
\end{remark}

Beyond the Riemannian setting, Coornaert \cite{mC93} studied groups acting on Gromov hyperbolic metric spaces, and constructed Patterson--Sullivan measures on the Gromov boundary. Coornaert and the second author \cite{CK} used this together with ideas from \cite{K1,K4} to prove that if a group $\Gamma$ acts properly and cocompactly by isometries on a proper geodesic Gromov hyperbolic metric space whose Gromov boundary contains at least $3$ points, then \eqref{eqn:unif} holds for the set $P(t)$ of conjugacy classes $[\gamma]$ of elements $ \gamma \in \Gamma$ such that $\inf_{x \in X} d(x, \gamma x) \le t$.\footnote{As with \cite{K4}, the upper bound in \cite{CK} has no $t$ in the denominator, but it could be added with some extra work.}

None of these results, however, establish the Margulis asymptotics \eqref{eqn:margulis} in any setting beyond negative curvature.\footnote{For rank $1$ manifolds, the asymptotics were announced in the thesis of Gunesch \cite{rG02}, but a complete proof was never published.}
For surfaces that have negative curvature outside of a collection of radially symmetric `caps', \eqref{eqn:margulis} was established by Weaver \cite{bW14}, following Margulis's original argument with suitable modifications. More recently, in an unpublished preprint \cite{Ri},
Ricks has announced the result in the rank 1 CAT(0) setting, which includes the nonpositive curvature case; he follows Margulis's approach, using geometric tools as in \cite{K4,K5} to compensate for the lack of uniform hyperbolicity. As we describe below, we take a similar approach in this paper, although we make a number of substantial changes that allow us to sidestep locations where we were unable to follow the argument in \cite{Ri}.


\subsection{Outline of the proof}\label{sec:outline}

The proof of Theorem \ref{thm:main} uses ergodic theoretic properties of the measure of maximal entropy $m$ for the geodesic flow $\phi^t$ on $SM$. 
Beyond the basic geometric properties that we review in \S\ref{sec:uniq} (and prove in Appendix \ref{sec:geom-pf}), the key ingredients are as follows.\footnote{Some of the notation here is simplified from the formal proofs, where we take greater care to indicate the dependence on various parameters, and to identify whether objects are defined in terms of $M$ itself or its universal cover.}
\begin{enumerate}[itemsep=1ex]
\item \emph{Flow box and slice (\S\ref{sec:box}):} small sets $B,S \subset SM$ with a product structure of (stable) $\times$ (unstable) $\times$ (flow), where $B$ and $S$ have the same cross-section in the first two components and have flow-lengths of $\epsilon$ and $\epsilon^2$, respectively.\footnote{Here we differ from prior approaches, which only used $B$.}
\item \emph{Closing lemma (\S\ref{sec:per-G}):} 
a nearly bijective correspondence between
\begin{itemize}
\item the set of $\epsilon$-segments in which some $c\in C(t)$ intersects $B$, where $C(t)$ is as in Theorem \ref{thm:equidist}, and
\item the set of components of $S\cap \phi^{-t} B$, indexed by a subset $\Gamma(t)$ of the fundamental group.
\end{itemize}
In particular, the measures $\nu_t$ from \eqref{eqn:nut} have
$\nu_t(B) \approx \frac \epsilon t \frac{\#\Gamma(t)}{\# C(t)}$.
\item \emph{Product structure and scaling properties (\S\S\ref{sec:PS} and \ref{sec:scaling}):} nearly every component of $S\cap \phi^{-t} B$ has measure $\approx e^{-ht} m(S)$, and thus $m(S\cap \phi^{-t} B) \approx \#\Gamma(t) e^{-ht} m(S)$;
\item \emph{Mixing property for $m$ (\S\ref{sec:asymp-gam}):} $m(S\cap \phi^{-t} B) \to m(S)m(B)$ as $t\to\infty$, so $\#\Gamma(t) \sim e^{ht} m(B)$.
\item \emph{Equidistribution (\S\ref{sec:equidist}):} $\nu_t\to m$ as in Theorem \ref{thm:equidist}, giving the asymptotic
$\# C(t) \sim \frac{\epsilon}t \frac{\# \Gamma(t)}{m(B)}
\sim \epsilon \cdot e^{ht}/t$.
\item \emph{Integration (\S\ref{sec:completion}):} Fixing $b>0$ and partitioning $(b,T]$ into $\epsilon$-intervals with endpoints $t_k$ gives a Riemann sum
\[
\# P(T) \sim \#(P(T) \setminus P(b)) \sim \sum_k \epsilon \cdot \frac{e^{ht_k}}{ t_k} \xrightarrow{\epsilon\to0}\int_b^T \frac 1t e^{ht}\,dt \sim \frac 1{hT} e^{hT}.
\]
\end{enumerate}

These are the same basic ingredients as in Margulis's original proof, with one notable exception: he works with components of $B \cap \phi^{-t} B$, which may have varying `depths' between $0$ and $\epsilon$, while we work with $S\cap \phi^{-t} B$ and can disregard components that do not have full depth $\epsilon^2$. 
(This also makes our use the mixing property in \S\ref{sec:mixing} somewhat different from prior approaches.) 
It should also be pointed out that because we do not have uniform hyperbolicity, we follow \cite{K4,Ri} and work with the boundary at infinity instead of stable and unstable manifolds: this shows up in the definition of the flow box and slice, in the closing lemma, and in the product structure and scaling properties. Moreover, we carry out many of the counting steps in terms of the fundamental group $\pi_1(M)$.

\section{Definitions and geometric preliminaries}\label{sec:uniq}

\subsection{Closed curves}\label{sec:closed-curves}

Throughout the paper, $(M,g)$ will be a closed Riemannian manifold.
A \emph{(smooth) closed curve} on $M$ is a smooth map $c \colon \R/\Z = S^1 \to M$. Two such curves $c_0,c_1 \colon \R/\Z = S^1\to M$ are \emph{freely homotopic} if there exists a continuous map $H\colon [0,1] \times \R/\Z \to M$ such that $H(0,t) = c_0(t)$ and $H(1,t) = c_1(t)$ for all $t$. This is an equivalence relation on the set of closed curves, whose equivalence classes we call \emph{free homotopy classes}.

By Cartan's Theorem \cite[Theorem 12.2.2]{dC92}, every free homotopy class of closed curves contains at least one closed geodesic, which minimizes length within the class. Thus \eqref{eqn:margulis} can be interpreted either as a result on the growth rate of free homotopy classes of closed geodesics, or of free homotopy classes of closed curves, since the two are equivalent. From now on we will work only with closed geodesics. 

Since it is customary to parametrize geodesics with unit speed, we will write a closed geodesic as $c\colon \R/\ell \Z \to M$, where $\ell>0$ is the \emph{length} of the closed geodesic. For manifolds without conjugate points (defined below), any two closed geodesics in the same free homotopy class have the same length; see \cite[Lemma 2.29]{CKW} for the (short) proof.

Given a closed geodesic $c$ and $n\in \N$, one can consider the closed geodesic $c^n \colon \R/n\ell\Z \to M$ that consists of $n$ orbits around $c$. A closed geodesic is \emph{primitive} if it cannot be written as $c^n$ for some closed geodesic $c$ and $n\geq 2$.  One can show that \eqref{eqn:margulis} is unaffected by whether we count all closed geodesics, or only primitive ones (see Remark \ref{rmk:apathy}). 

As we will see in \S\ref{sec:|gamma|} below, there is a one-to-one correspondence between free homotopy classes of closed geodesics and conjugacy classes in the fundamental group. A conjugacy class is primitive if its elements cannot be written as proper powers. Note that it is possible to have a primitive closed geodesic whose free homotopy class corresponds to a non-primitive conjugacy class in the fundamental group: the boundary circle of a M\"obius strip provides an example.

\subsection{Geodesic flows and standing assumptions}

Let $SM$ be the unit tangent bundle of $M$, and $\pi\colon SM \to M$ the canonical projection that takes each tangent vector to its basepoint.  Each $v\in SM$ determines a unique geodesic $c_v\colon \R \to M$ with $c_v(0) = \pi(v)$ and $\dot{c}_v(0) = v$.  The \emph{geodesic flow} on $M$ is the flow $\phi^t$ on $SM$ given by
\begin{equation}\label{eqn:geodesic-flow}
\phi^t(v) = \dot{c}_v(t).
\end{equation}
Dynamical properties of the geodesic flow are related to curvature properties of the manifold, and in particular to the behavior of Jacobi fields, which govern how nearby geodesics evolve in time.  Formally, a \emph{Jacobi field} along a geodesic $c\colon \R\to M$ is a vector field $J(t)$ along $c$ satisfying $J'' + R(J,\dot c) \dot c = 0$, where $R$ is the Riemann curvature tensor and $'$ is covariant differentiation along $c$.  Two points $p,q$ on $c$ are \emph{conjugate} if there is a nonzero Jacobi field along $c$ that vanishes at both $p$ and $q$.  

Throughout the paper, we consider manifolds without conjugate points.  This class includes all manifolds with nonpositive curvature, as well as many manifolds with positively curved regions \cite{rG75}.  It is also characterized by the property that the exponential map is nonsingular, or equivalently, that any two points in the universal cover are connected by a unique geodesic.

We impose two more standing assumptions:\footnote{Note that these assumptions rule out some higher-dimensional manifolds with nonpositive curvature.} we assume that $M$ admits another Riemannian metric $g_0$ for which all sectional curvatures are negative, and that $M$ satisfies the divergence property in Definition \ref{def:div} below. In the following sections we gather various definitions and results about such manifolds; some of the lemmas are quoted directly from the literature; the remainder are proved in Appendix \ref{sec:geom-pf}.

\subsection{Visibility manifolds and the boundary at infinity}

Write \(X\) for the universal cover of $M$ and \(SX\) for the unit tangent bundle of \(X\) with respect to the lift of the Riemannian metric \(g\).  Let $\pr\colon X\to M$ be the canonical projection; we write $\pr_* \colon SX\to SM$ for the map this induces on the unit tangent bundles.
From now on we will use an underline to denote objects in $M$ and $SM$, so that for example $\hc$ will denote a geodesic on $M$, and $\tc$ a geodesic on $X$ that lifts $\hc$, that is, $\hc = {\pr}\circ\tc$. We write $\phi^t$ for the geodesic flow on both $SM$ and $SX$.

Following \cite{Eb1}, we can define the \emph{boundary at infinity} $\partial X$ provided the following condition is satisfied.

\begin{definition}[{\cite[Definition 1.3]{Eb1}}]\label{def:vis}
A simply connected Riemannian manifold $X$ without conjugate points 
is a \emph{uniform visibility manifold} if for every $\epsilon>0$ there exists $L>0$ such that whenever a geodesic $\tc\colon [a, b] \to X$ stays at distance at least $L$ from some point $p\in X$, then the angle sustained by $\tc$ at $p$ is less than $\epsilon$, that is
  \begin{equation*}
    \measuredangle_p(\tc)=\sup_{a\leq s,t\leq b} \measuredangle_p(\tc(s),\tc(t))<\epsilon.
  \end{equation*}
If $M$ is a Riemannian manifold without conjugate points whose universal cover $X$ is a uniform visibility manifold, then we say that $M$ is a uniform visibility manifold.
\end{definition}

The uniform visibility property implies the following \emph{divergence property} \cite[Proposition 1.5]{Eb1}, and when $M$ admits a negatively curved ``background'' metric $g_0$, the two conditions are equivalent \cite[Theorem 5.1]{Eb1}.\footnote{The divergence property is not mentioned in the statement of \cite[Theorem 5.1]{Eb1}, but the ``Added in proof'' section of that paper discusses its necessity.}

\begin{definition}\label{def:div}
The manifold $(M,g)$ has the \emph{divergence property} if given any geodesics $\tc_1 \neq \tc_2$ in the universal cover with $\tc_1(0) = \tc_2(0)$, we have $\lim_{t\to\infty} d(\tc_1(t),\tc_2(t)) = \infty$.
\end{definition}

For the remainder of the paper we assume that $(M,g)$ is has no conjugate points, admits a background metric of negative curvature, and satisfies the divergence property (and thus the uniform visbility property as well).

\begin{remark}
In dimension 2, every surface of genus at least 2 without conjugate points admits a negatively curved metric, and the divergence property holds by \cite{Gr}.
In higher dimensions the corresponding result is no longer true: there are rank 1 manifolds of nonpositive curvature, such as the graph manifolds described by Gromov \cite[\S6]{K5}, that admit totally geodesic isometric embeddings of $\R^2$ and thus fail the uniform visibility condition.
\end{remark}

We describe a compactification of $X$ following Eberlein \cite{Eb1}.  

\begin{definition}
Two geodesic rays  $\tc_1, \tc_2\colon [0,\infty) \to X$ are called \emph{asymptotic}
if $ \sup_{t\geq 0} d(\tc_1(t), \tc_2(t)) < \infty$.  This is an equivalence relation; we write  $ \partial X$ for the set of equivalence classes and call its elements \emph{points at infinity}. We denote the equivalence class of a geodesic ray (or geodesic) $\tc$  by $\tc(\infty)$.
\end{definition}

Given $\tv\in SX$, let $\tc_{\tv}$ be the unique geodesic with $\dot{\tc}_{\tv}(0)=\tv$.  The following is proved in \cite[Propositions 1.5 and 1.14]{Eb1}.

\begin{lemma}\label{lem:p-to-xi}
Given any $p\in X$ and $\xi\in \partial X$, there is a unique geodesic ray $\tc=\tc_{p,\xi}\colon [0,\infty)\to X$ with $\tc(0) = p$ and $\tc(\infty) = \xi$.
Equivalently, the map  $f_p\colon S_pX \to \partial X$ defined by $f_p(\tv) = \tc_{\tv}(\infty)$ is a bijection. Moreover, for a given $\xi$, the initial vector $\dot{\tc}_{p,\xi}(0)$ depends continuously on $p$.
\end{lemma}


Following Eberlein we equip $\partial X$ with a topology that makes it a compact metric space homeomorphic to $S^{n-1}$.  Fix $p\in X$ and let $f_p\colon S_pX  \to \partial X$
 be the bijection $\tv \mapsto \tc_{\tv}(\infty)$ from Lemma \ref{lem:p-to-xi}. The topology (sphere-topology) on $\partial X$ is defined such that $f_p$
becomes a homeomorphism.
Since the map $f_q^{-1} f_p \colon S_p X \to S_q X$ is a homeomorphism for all $q\in X$ \cite{Eb1}, the topology is independent of the reference point $p$.

The topologies on $\partial X$ and $X$
extend naturally
 to  $\bar X: =  X \cup \partial X$
by requiring that the map
$\varphi\colon B_1(p) = \{\tv \in T_p X:   \|\tv\| \le 1\} \to \bar X$
defined by
\[
\varphi(\tv) = \begin{cases}
 \exp_p\left(\frac{\tv}{1-\|\tv\|}\right) & \|\tv\| < 1\\
f_p(\tv) & \|\tv\| = 1
\end{cases}
\]
is a homeomorphism. This topology, called the cone topology, was introduced by Eberlein
and O'Neill \cite{EO} in the case of Hadamard manifolds and by Eberlein   \cite{Eb1} in the case of visibility manifolds. In
particular, $\bar X$ is homeomorphic to a closed ball in
$\mathbb{R}^n$.
The relative topology on $\partial X$ coincides with the sphere topology, and the relative topology on $X$ coincides with the manifold topology.

\begin{remark}\label{rem:minmal}
The manifold $M$ is the quotient of $X$ by the isometric action of the group $\Gamma$ of deck transformations, which is isomorphic to the fundamental group $\pi_1(M)$.  
This extends to a continuous action on $\partial X$.
The geodesic flow is topologically transitive \cite{Eb1}, so every $\Gamma$-orbit in $\partial X$ is dense, i.e., the action on  $X$ is minimal.
Moreover, the Morse lemma (see for example \cite[Theorem 2.3]{K2})
gives a topological conjugacy between the actions of $\Gamma$ on the boundaries at infinity with respect to the two metrics $g$ and $g_0$; in particular, every element $\gamma\in \Gamma \setminus \id$ has exactly two fixed points on $\ideal$, since this is true for $g_0$.
\end{remark}

\begin{definition}\label{def:gamma-pm}
Given $\gamma\in\Gamma\setminus\id$, let $\xi_\gamma^-\in \ideal$ be the repelling fixed point of $\gamma$, and $\xi_\gamma^+ \in \ideal$ be the attracting fixed point of $\gamma$.
\end{definition}

Every geodesic $\tc\colon \R \to X$ determines two distinct points $\tc(-\infty)$ and $\tc(\infty)$ on $\ideal$.  We will also use the notation
\begin{equation}\label{eqn:vpm}
\tv^- := \tc_{\tv}(-\infty)
\quad\text{and}\quad
\tv^+ := \tc_{\tv}(\infty).
\end{equation}
It is useful to write
\begin{equation}\label{eqn:square}
\sqbd := \{ (\xi,\eta) \in (\ideal)^2 : \xi \neq \eta\}
\end{equation}
and to consider the \emph{endpoint map}
\begin{equation}\label{eqn:endpoints}
\begin{aligned}
E \colon SX &\to \sqbd, \qquad
\tv &\mapsto 
(\tv^-,\tv^+).
\end{aligned}
\end{equation}
By \cite[Proposition 1.7]{Eb1}, every pair of distinct points on $\ideal$ is joined by at least one geodesic, so the map $E$ is onto.  Note that $E$ is continuous by the definition of the topology on $\ideal$.

It is sometimes important to work with pairs of points at infinity that are joined by \emph{exactly} one geodesic.

\begin{definition}\label{def:expansive}
The \emph{expansive set} of $X$ is
\begin{equation}\label{eqn:expansive}
\cal E := \{ \tv \in SX \mid \phi^\R (\tv) = E^{-1}(\tv^-,\tv^+) \}.
\end{equation}
We say that an invariant Borel probability measure $\mu$ on $SM$ is \emph{almost expansive} if $\mu(\pr_*\cal{E})=1$.
\end{definition}

\subsection{Periodic orbits and deck transformations}\label{sec:|gamma|}

There is a one-to-one correspondence between closed geodesics on $M$ and conjugacy classes in the group $\Gamma$ of deck transformations. Here we recall some of the elements of this correspondence that we will need.

Given a closed geodesic $\hc\colon \R/\ell \Z\to M$, the length $\ell>0$ is such that $\hc(t+\ell) = \hc(t)$ for all $t\in\R$, and thus for every lift $\tc$ of $\hc$ there is a unique $\gamma\in\Gamma$ such that 
\begin{equation}\label{eqn:axis}
\tc(t+\ell) = \gamma \tc(t) \text{ for all }t\in\R.
\end{equation}

\begin{definition}
If $\tc$, $\gamma$, and $\ell>0$ are such that \eqref{eqn:axis} holds, then we say that $\gamma$ is the \emph{axial isometry} of $(\tc,\ell)$, and $\tc$ is an \emph{axis} of $\gamma$.
\end{definition}

If $\tc$ is an axis of some $\gamma$, then it follows immediately that $\hc = {\pr}\circ\tc$ is a closed geodesic on $M$. Moreover, every axial isometry $\gamma$ of $\tc$ (with any value of $\ell$) fixes $\tc(\pm\infty)$, and by Remark \ref{rem:minmal} and Definition \ref{def:gamma-pm} this proves the following lemma.

\begin{lemma}\label{lem:endfix}
If $\tc$ is an axis of $\gamma$, then
$\tc(-\infty) = \xi_\gamma^-$ and $\tc(\infty) = \xi_\gamma^+$.
\end{lemma}

\begin{remark}
There may be multiple geodesics connecting $\xi_\gamma^\pm$ besides $\tc$ itself; these geodesics may or may not be axes of $\gamma$.
\end{remark}

It follows from Preissman's Theorem that in the background metric $g_0$, given every pair of points on $\ideal$, the set of $\gamma$ that fixes these two points is either trivial or an infinite cyclic subgroup \cite[Lemma 12.3.5]{dC92}. Using the topological conjugacy from Remark \ref{rem:minmal} we get the same result for $g$.

\begin{lemma}\label{lem:Preissman}
Given any geodesic $\tc$ on $X$ such that $\hc = {\pr}\circ \tc$ is closed, the set of $\gamma\in \Gamma$ fixing $\tc(-\infty)$ and $\tc(\infty)$ is an infinite cyclic subgroup.
\end{lemma}

\begin{remark}
It is possible for the subgroup in Lemma \ref{lem:Preissman} to contain non-identity elements for which neither $\tc$ nor its reverse is an axis; for example, this occurs when $\tc$ lifts a boundary circle of a M\"obius strip.
\end{remark}

The preceding discussion lets us go from closed geodesics to deck transformations. We will also need to go in the other direction.

\begin{definition}
Given $\gamma\in\Gamma\setminus\id$, the \emph{length} of $\gamma$ is
\[
|\gamma|:=\inf\{d(q,\gamma q): q\in X\}.
\]
\end{definition}

\begin{lemma}\label{lem:ax}
For every $\gamma\in\Gamma\setminus\id$, there exists $q_0\in X$ such that $d(q_0,\gamma q_0) = |\gamma| \geq 2\inj(M)$. Moreover, if $\tc\colon \R\to X$ is a geodesic joining $q_0$ and $\gamma q_0$, then $\tc(t+|\gamma|) = \gamma\tc(t)$ for all $t\in\R$, so
$\hc := {\pr}\circ \tc$ is a closed geodesic.
\end{lemma}

Lemma \ref{lem:ax} is standard: the idea of the proof is in \cite[\S11.7, Theorem 10]{BC64}, and a complete proof is in \cite[pages 196--197]{tS96}.
We need two more standard results
relating free homotopy classes, axial isometries, and endpoints on $\ideal$; for completeness we provide proofs in Appendix \ref{sec:geom-pf}. 

\begin{lemma}\label{lem:hom-lift}
Let $\hc_0 \colon \R/\ell_0\Z \to M$ and $\hc_1 \colon \R/\ell_1\Z \to M$ be closed geodesics on $M$ that lie in the same free homotopy class.\footnote{Here we specify the periods $\ell_0,\ell_1$ explicitly, as in \S\ref{sec:closed-curves}; notice that replacing $\ell_j$ with $n\ell_j$ would change the free homotopy class.} Then given any lift $\tc_0$ of $\hc_0$, there is a lift $\tc_1$ of $\hc_1$ such that $(\tc_0,\ell_0)$ and $(\tc_1,\ell_1)$ have the same axial isometry, and hence the same endpoints.
\end{lemma}

\begin{lemma}\label{lem:2-axes}
If $\gamma$ is the axial isometry for both $(\tc_0,\ell_0)$ and $(\tc_1,\ell_1)$, then $\ell_0 = \ell_1 = |\gamma|$, and the corresponding closed geodesics lie in the same free homotopy class.
\end{lemma}

We remark that these lemmas (apart from the claim about endpoints on $\ideal$) are true for general closed manifolds, although for convenience our proof of Lemma \ref{lem:2-axes} uses $\ideal$ via Lemma \ref{lem:endfix}.

\subsection{Other geometric hypotheses}\label{sec:geom-hyp}

In \cite{CKW} we proved existence and uniqueness of the measure of maximal entropy for a geodesic flow on a closed Riemannian manifold $M$ without conjugate points under the following conditions.
\begin{enumerate}
\item There exists a Riemannian metric $g_0$ on $M$ for which all sectional curvatures are negative.
\item The visibility axiom is satisfied.\footnote{The result in \cite{CKW} is formulated using the divergence property, but they are equivalent when there is a negatively curved background metric.}
\item The fundamental group $\pi_1(M)$ is \emph{residually finite}:  the intersection of its finite index subgroups is trivial.
\item There exists $h_0 < h$ such that 
any ergodic invariant Borel probability measure $\mu$ on $SM$ with entropy $>h_0$ 
is almost expansive.
\end{enumerate}
As discussed in \cite[Remarks 2.1 and 3.2]{CKW}, the first of these conditions is a genuine topological restriction, and excludes certain manifolds of nonpositive curvature that still obey the Margulis asymptotics by \cite{Ri}. It is not known whether or not the remaining three conditions can fail if the first holds.  All four conditions hold for every surface of genus $\geq 2$ \cite[\S3]{CKW}.

In addition to proving existence and uniqueness of the MME, we proved in \cite{CKW} that it is mixing and has a product structure and scaling properties that we describe in the next section.

\section{Scaling properties and product structure}\label{sec:box}

\subsection{Busemann functions and conformal densities}\label{sec:PS}

\begin{lemma}[{\cite[Proposition 1]{jE77}}]
Given $\tv\in SX$ and $q\in X$, the limit
\[
b_{\tv}(q) := \lim_{t\to\infty} \big( d(q,\tc_{\tv}(t)) - t\big)
\]
exists and defines a $C^1$ function on $X$.  Moreover, we have 
\[
\grad b_{\tv}(q) = \lim_{t\to\infty} \grad \big( d(q,\tc_{\tv}(t)) - t \big).
\]
\end{lemma}

The function $b_{\tv}$ is called the \emph{Busemann function} associated to $\tv$. It is in fact $C^{1,1}$ \cite[Satz 3.5]{gK86}.

\begin{definition}\label{def:busemann}
Given $p\in X$ and $\xi \in \partial X$, let $\tv\in S_p X$ be the unique unit tangent vector at $p$ such that $\tc_{\tv}(\infty) = \xi$.  
We call $b_\xi(q,p) := b_{\tv}(q)$ the Busemann function based at $\xi$ 
and normalized by $p$ ($b_\xi(p,p) =0$).
\end{definition}

The zero set of the Busemann function is the horosphere through $p$ centered at $\xi$, say \(H_\xi(p)\), and $b_\xi(q,p)$ can be interpreted as the distance you need to travel along the geodesic from $q$ towards $\xi$ in order to reach this horosphere; see Figure \ref{fig:b-beta}. 
With this interpretation the following \emph{cocycle property} of the Busemann function becomes transparent:
\begin{equation}\label{eqn:pqr}
b_\xi(p,q) = b_\xi(p,r) + b_\xi(r,q)
\text{ for all } \xi\in \ideal\text{ and } p,q,r\in X.
\end{equation}
Note also that
\begin{equation}\label{eqn:pq}
b_\xi(p,q) = -b_\xi(q,p) \text{ for all } \xi\in\ideal\text{ and } p,q\in X.
\end{equation}
Figure \ref{fig:b-beta} also illustrates
the function $\beta_p \colon \sqbd \to (0,\infty)$ defined by
\begin{equation}\label{eqn:beta-p}
\beta_p(\xi, \eta) = - (b_\xi(q, p) + b_\eta(q, p)),
\end{equation}
where $q$ is a point on a geodesic $\tc$ connecting $\xi$ and $\eta$.
Geometrically, $\beta_p(\xi, \eta)$ is the length of the
segment $\tc$ which is cut out by the horoballs through
$(p, \xi)$ and $(p, \eta)$.
Since $\grad_q b_\xi(q,p) =- \grad_q b_\eta(q,p)$ 
for all points $q$ on geodesics
connecting $\xi$ and $\eta$, this number is independent of the choice of $q$.

\begin{figure}[htbp]
\begin{tikzpicture}[scale=1.8,baseline=0]
\draw[red] (-.6,0) circle(.4);
\coordinate (q) at ({-1+sqrt(3)/2},-.5);
\draw[blue,dotted] (q) arc (30:90:1);
\draw[blue,thick] (q) arc (30:47:1);
\fill (q) node[below right]{$q$} circle(1pt);
\draw (0,0) circle(1);
\fill (-1,0) node[left]{$\xi$} circle(1pt);
\fill (-.2,0) node[above right]{$p$} circle(1pt);
\draw[dashed,->] (.83,-.63) node[right]{$b_\xi(q,p)=\pm d(q,H_\xi(p))$} -- (-.18,-.33);
\draw (-.6,.1) node{$-$} (.6,.1) node{$+$};
\end{tikzpicture}
\hfill
\begin{tikzpicture}[scale=1.8,baseline=0]
\draw[red] (.4,0) circle(.6);
\coordinate (q) at ({1-sqrt(2)/2},{-1+sqrt(2)/2});
\draw[dotted] (0,-1) arc (180:90:1);
\draw[thick] (q) arc (135:125:1);
\draw[thick] (q) arc (135:152:1);
\draw[blue] (0,-1) arc (-90:270:0.52);
\draw (0,0) circle(1);
\fill (1,0) node[right]{$\xi$} circle(1pt);
\fill (0,-1) node[below]{$\eta$} circle(1pt);
\fill (-.2,0) node[left]{$p$} circle(1pt);
\draw[dashed,->] (.4,.15) node[above]{$\beta_p(\xi,\eta)$} -- (.25,-.25);
\end{tikzpicture}
\caption{Geometric interpretation of $b_p(q,\xi)$ and $\beta_p(\xi,\eta)$.}
\label{fig:b-beta}
\end{figure}
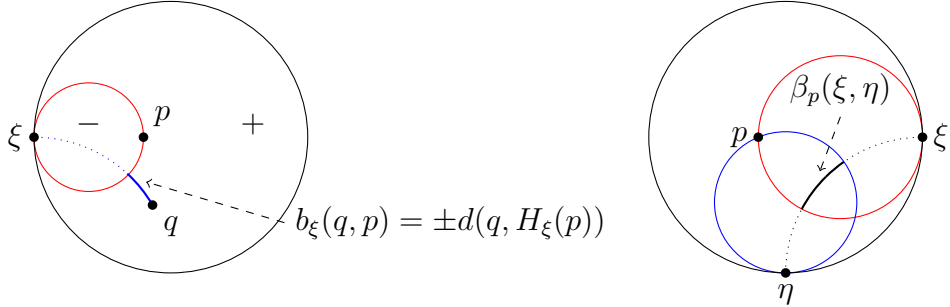

It is easy to see that $b_\xi(p,q)$ is $1$-Lipschitz in both $p$ and $q$ by either using the definition directly or by using \eqref{eqn:pqr} and the fact that $|b_\xi(p,q)| \leq d(p,q)$. We will also need to know how $b_\xi(p,q)$ varies with $\xi$. This is equivalent to understanding how $H_\xi(p)$ varies with $\xi$, which is accomplished by the following result; see \cite[Proposition 6.2]{CKW} for the statement we use and for precise references to the proof, which can be found in \cite{jP77} and uses results from \cite{Eb1} (see also \cite[Lemma 4.11]{rR07}).

\begin{proposition}\label{prop:conti}
Let \((M,g)\) be a closed Riemannian manifold without conjugate points that satisfies the uniform visibility condition and admits a background metric of negative curvature, and let $X$ be its universal cover. Then for every $p\in X$, the map $\xi\mapsto H_\xi(p)$ is continuous in the following sense: if $\xi_n \to \xi\in \ideal$ and $K\subset X$ is compact, then $H_{\xi_n}(p) \cap K \to H_\xi(p) \cap K$ uniformly in the Hausdorff topology.
\end{proposition}

\begin{corollary}\label{cor:bus-cts}
The functions $(v,q) \mapsto b_v(q)$ and $(\xi,p,q) \mapsto b_\xi(p,q)$ are continuous on $SX \times X$ and $\ideal \times X \times X$, respectively.
\end{corollary}

Now we can define the types of measures that we work with to carry out the counting argument.

\begin{definition}\label{def:conformal} 
Given $h>0$, an \emph{$h$-conformal density} on $\partial X$ is a family of finite measures $\{\mu_p\}_{p \in X}$ on $\partial X$ with the following properties.
\begin{enumerate}
\item\label{PS-a} $\supp \mu_p = \partial X$ for all $p \in X$.
\item\label{PS-b} $\{\mu_p\}_{p \in X}$ is $\Gamma$-equivariant: for all Borel sets $A \subset \partial X$, we have
\[
\mu_{\gamma p} (\gamma A) = \mu_p(A).
\]
\item\label{PS-c} $\frac{d \mu_q}{d\mu_p} (\xi) = e^{-h b_\xi(q, p)}$
for almost all $\xi  \in \partial X$.
\end{enumerate}
\end{definition}

In the setting of Theorem \ref{thm:main}, an $h$-conformal density was obtained in \cite[Proposition 5.1]{CKW} via a Patterson--Sullivan construction, where $h$ is the topological entropy of the geodesic flow.  Given a conformal density,
the proof of \cite[Lemma 2.4]{K5} gives the following.

\begin{lemma}\label{5.4.A}
For $p \in X$, the measure $\bar\mu$ on $\sqbd \subset (\ideal)^2$ defined by
\begin{equation}\label{eqn:barmu}
d \bar\mu (\xi, \eta) = e^{h\beta_p(\xi,\eta)}\,
 d\mu_p(\xi)\, d\mu_p(\eta)
\end{equation}
is $\Gamma$-invariant.
\end{lemma}

Under the hypotheses of Theorem \ref{thm:main}, we proved in \cite[Theorem 5.6 and Remark 5.7]{CKW}
that $\bar\mu$-a.e.\ pair $(\xi,\eta) \in \sqbd$ is connected by a unique geodesic $V(\xi,\eta)$, and that writing $\lambda_{\xi,\eta}$ for Lebesgue measure along the trajectory corresponding to $V(\xi,\eta)$, the measure defined on $SX$ by
\begin{equation}\label{eqn:Kaim}
\tm(A) = \int_{\sqbd} \lambda_{\xi,\eta}(A) \,d\bar\mu(\xi,\eta)
\end{equation}
is related to the unique MME $\hm$ on $SM$ by
\begin{equation}\label{eqn:hmtm}
\tm(A) = \int_{SM} \#(\pr_*^{-1}(\hv) \cap A) \,d\hm(\hv).
\end{equation}
Moreover, in \cite[Theorem 6.1 and Remark 6.2]{CKW} we used an argument of Babillot \cite{Bab} to prove that the flow is mixing with respect to $\hm$ under these same assumptions.

\subsection{Local product structure near expansive vectors}\label{sec:lps}

From now on we fix an expansive vector $\tv_0 \in \mathcal{E}$ (which is nonempty since it has full measure for $\tm$), and let $p = \pi(\tv_0) \in X$.  This will be a `reference point' for all the definitions that follow, and we will suppress $\tv_0$ and $p$ from the notation.  We fix a scale $\epsilon \in (0,\min(\frac 18, \frac{\inj (M)}{4})]$, where $\inj(M)$ is the injectivity radius
of $M$.
At the very end of the proof, in \S\ref{sec:completion}, we will need to take a limit as $\epsilon\to 0$, but until then $\epsilon$ will be fixed.

\begin{definition}\label{def:hopf}
The \emph{Hopf map} $H \colon SX \to \sqbd \times \R$ for $p\in X$ is 
\begin{equation}\label{eqn:hopf}
H(\tv) := (\tv^-, \tv^+, s(\tv)),
\text{ where }
s(\tv) :=b_{\tv^-}(\pi \tv, p).
\end{equation}
\end{definition}

The following fundamental fact will be useful:
\begin{equation}\label{eqn:+t}
s(\phi^t \tv) = s(\tv) + t
\text{ for all } \tv\in SX \text{ and } t\in \R.
\end{equation}

Note that the Hopf map is continuous by Corollary \ref{cor:bus-cts} and the definition of the topology on $\ideal$.
Following Ricks \cite{Ri}, we use the Hopf map to define a local product structure on a neighborhood in $SX$.  Given disjoint sets $\Pa,\Fu\subset \ideal$, the set $H^{-1}(\Pa\times\Fu\times\{0\})$ represents the set of all $\tv\in SX$ whose past history under $\phi^t$ is given by $\Pa$, whose future evolution is given by $\Fu$, and 
such that $b_{v^-}(\pi v, p) = 0$.

\begin{lemma}\label{lem:cpt}
For any disjoint closed sets $\Pa,\Fu\subset \ideal$, the set $H^{-1}(\Pa\times\Fu\times\{0\})\subset SX$ is compact.
\end{lemma}

We will use the following choice of $\Pa,\Fu$: given our fixed choice of $\tv_0 \in S_p X \cap \mathcal{E}$, we consider for each $\theta>0$ the sets
\begin{equation}\label{eqn:Cpm}
\begin{aligned}
\Pa = \Pa_\theta &:= \{ \tw^- : \tw\in S_p X \text{ and } \measuredangle_p(\tw,\tv_0) \leq \theta \}, \\
\Fu = \Fu_\theta &:= \{ \tw^+ : \tw\in S_p X \text{ and } \measuredangle_p(\tw,\tv_0) \leq \theta \}.
\end{aligned}
\end{equation}
The function $\theta \mapsto \bar\mu(\Pa_\theta\times\Fu_\theta)$ is nondecreasing and thus has at most countably many discontinuities; from now on we assume that $\theta$ is chosen to be a point of continuity of this function, so that
\begin{equation}\label{eqn:cty}
\lim_{\rho\to\theta} \bar\mu(\Pa_\rho\times\Fu_\rho) = \bar\mu(\Pa_\theta\times\Fu_\theta).
\end{equation}
The assumption that $\tv_0 \in \mathcal{E}$ lets us strengthen the boundedness result in Lemma \ref{lem:cpt}.

\begin{lemma}\label{lem:inj0}
Let $\tv_0,p,\epsilon$ be as above.  Then there exists $\theta_1 >0$ such that for all $0 < \theta \le \theta_1$ we have $\diam \pi H^{-1}(\Pa\times\Fu\times\{0\}) < \frac\epsilon2$.
\end{lemma}

Given $\alpha \in (0,\frac 32\epsilon]$, we consider the \emph{flow box}
\begin{equation}\label{eqn:B}
\begin{aligned}
\tB = \tB_\theta^\alpha &:= H^{-1}(\Pa\times\Fu \times [0, \alpha]) \\
&= \{\tw\in SX : (\tw^-,\tw^+) \in \Pa\times\Fu \text{ and } s(\tw) \in [0,\alpha]\} \\
&= \bigcup\{\phi^{[0, \alpha]}\tw  :  (\tw^-, \tw^+) \in \Pa\times \Fu \text{ and } s(\tw) = 0\},
\end{aligned}
\end{equation}
which is a union of tangent vectors to geodesic rays of length $\alpha$; note that each trajectory of $\phi^t$ that enters $\tB$ must do so through the ``front face'' $H^{-1}(\Pa\times\Fu\times\{0\})$.

We will occasionally write $\tB^\alpha$ or $\tB_\theta$ when only one of $\alpha,\theta$ needs to be explicitly specified (the other being constant through a long portion of the proof).
In fact we will only need to consider the cases $\alpha = \epsilon$, $\alpha = \epsilon \pm 4\epsilon^2$, and $\alpha = \epsilon^2$. The last of these is important enough to deserve its own notation, and we write
\begin{equation}\label{eqn:S}
\tS = \tS_\theta := \tB_\theta^{\epsilon^2} = H^{-1}(\Pa\times\Fu\times[0,\epsilon^2]).
\end{equation}
Thus from now on we will always use $\tB$ to denote a \emph{box} with depth $\alpha \approx \epsilon$, and $\tS$ to denote the \emph{slice} with depth exactly $\epsilon^2$. We point out that since $0<\epsilon\leq \frac 18$, we have $\epsilon + 4\epsilon^2 \leq \frac 32\epsilon$ and $\epsilon - 4\epsilon^2 \geq \frac \epsilon2$.
We also observe that by \eqref{eqn:cty} and the product structure of $\tS$ and $\tB$, we have
\begin{equation}\label{eqn:SB-cty}
\lim_{\rho\to\theta} \tm(\tS_\rho) = \tm(S_\theta)
\quad\text{and}\quad
\lim_{\rho\to\theta} \tm(\tB_\rho^\alpha) = \tm(B_\theta^\alpha)
\end{equation}
for every $\alpha$, and also
\begin{equation}\label{eqn:bdry-null}
\tm(\partial \tB_\theta^\alpha) = 0.
\end{equation}
The following is an immediate consequence of Lemma \ref{lem:inj0}.

\begin{lemma}\label{lem:inj}
Let $\tv_0,p,\epsilon$ be as above, and let $\theta_1$ be as in Lemma \ref{lem:inj0}. Then for all $0<\theta\leq \theta_1$ and $\alpha \leq \frac32\epsilon$, we have $\diam \pi \tB_\theta^\alpha < 2\epsilon$.
\end{lemma}

The following lemma 
will be used  in Lemmas \ref{lem:b-close} and \ref{lem:scaling}; see Appendix \ref{sec:geom-pf} for the proof.

\begin{lemma}\label{lem:eps'}
Given $\tv_0,p,\epsilon>0$ as above, there exists $\theta_2>0$ such that for all $0 < \theta\leq \theta_2$, given any $\xi,\eta\in \Pa_\theta$ and any $q$ lying within $2\epsilon$ of $\pi H^{-1}(\Pa_\theta\times \Fu_\theta\times [0,\infty))$,
we have $|b_\xi(q,p) - b_\eta(q,p)| < \epsilon^2$, with a similar estimate when the roles of $\Pa_\theta$ and $\Fu_\theta$ are reversed.
\end{lemma}

Let $\theta_0 = \min(\theta_1,\theta_2)$; from now on we will always consider $0<\theta \leq \theta_0$ so that Lemmas \ref{lem:inj0}, \ref{lem:inj}, and \ref{lem:eps'} all hold.

Using Lemma \ref{lem:inj} and the fact that $\epsilon < \frac 14 \inj(M)$, we see that the quotient map $\pr\colon X\to M$ is injective on $\pi\tB$, and similarly $\pr_*\colon SX\to SM$ is injective on $\tB$. We will write
\begin{equation}\label{eqn:hB}
\hB = \hB_\theta^\alpha := \pr_*(\tB_\theta^\alpha)
\text{ and }
\hS = \hS_\theta := \pr_*(\tS_\theta)
\end{equation}
for the flow box and the slice in $SM$.

\section{Counting with the fundamental group}\label{sec:per-G}

\subsection{An upper counting bound}\label{sec:upper-bound}

In \S\ref{sec:|gamma|}, we described the relationship between closed geodesics and deck transformations $\gamma\in\Gamma$.
Now we return to this question, restricting our attention to closed geodesics that ``pass through'' $\hB \subset SM$ in the sense that $\dot{\hc}(s) \in \hB$ for some $s$.
In particular, we estimate $\#C(t)$ from Theorem \ref{thm:equidist} in terms of $\nu_t(\hB)$ (recall \eqref{eqn:nut}) and a certain subset of $\Gamma$.  In \S\ref{sec:mixing} we will use the scaling and mixing properties of $\hm$ to estimate the size of this subset; then in \S\ref{sec:fin} we will show that $\nu_t\to \hm$ and combine all of these estimates to complete the proof.

To describe the subset of $\Gamma$ that we work with, start by recalling from Theorem \ref{thm:equidist} that $C(t)$ is any maximal set of pairwise non-free-homotopic closed geodesics in $M$ with lengths in $(t-\epsilon,t]$, and  that  \eqref{eqn:nut} gives
\begin{equation}\label{eqn:Ct-nut}
\#C(t) = \frac{\sum_{\hc\in C(t)} \Leb_{\hc}(\hB^\alpha_\theta)}{t\nu_t(\hB^\alpha_\theta)}
\text{ for every }\alpha,\theta.
\end{equation}
By the definition of $\hB_\theta^\alpha$, every $\hv\in \hB_\theta^\alpha$ has the property that the connected component of $0$ in $\{s\in\R : \phi^s \hv \in \hB_\theta^\alpha\}$ is an interval of length $\alpha$, which begins when $\phi^s\hv \in \pr_*H^{-1}(\Pa\times\Fu\times\{0\})$. Let $\Pi(t)$ be the set of tangent vectors initiating one of these segments; that is,
\begin{equation}\label{eqn:Pit}
\Pi(t) := \{ \dot{\hc}(s) \in \pr_*H^{-1}(\Pa\times\Fu\times\{0\}) : \hc\in C(t), s\in \R \}.
\end{equation}
The set $\Pi(t)$ is finite and does not depend on $\alpha$. From \eqref{eqn:Ct-nut} we get
\begin{equation}\label{eqn:Ct-nut-1}
\#C(t) = \frac{\alpha}t \frac{ \#\Pi(t)}{ \nu_t(\hB_\theta^\alpha)}.
\end{equation}
We estimate $\#\Pi(t)$ by associating to each $\hv\in \Pi(t)$ the axial isometry of an appropriate lift.


\begin{definition}\label{def:Theta}
Given $\hv\in \Pi(t)$, let $\ell=\ell(\hv) \in (t-\epsilon,t]$ be such that $\phi^\ell \hv = \hv$, and let $\tv$ be the unique lift of $\hv$ such that $\tv\in \tB_\theta^\alpha$.  Define $\Theta(\hv) \in \Gamma$ to be the axial isometry of $(c_{\tv},\ell)$; that is, the
unique isometry of $X$ such that $\phi^\ell\tv = \Theta(\hv)_* \tv$. Observe that $|\Theta(\hv)| = \ell$.
\end{definition}

Observe that for each $\hv\in \Pi(t)$ and $\gamma = \Theta(\hv)$, we have $\phi^t\tv = \phi^{t-\ell} \gamma_* \tv \in \gamma_* \tB_\theta^\epsilon$, and so $\tv \in \tS_\theta \cap \phi^{-t} \gamma_* \tB_\theta^\epsilon$.
With this in mind, we define for each $t>0$ and $\alpha \in (0,\frac 32 \epsilon]$ the set
\begin{equation}\label{eqn:Gt}
\Gamma(t,\alpha) 
= \Gamma_\theta(t,\alpha) :=
\{\gamma\in \Gamma : \tS_\theta \cap \phi^{-t} \gamma_* \tB_\theta^\alpha \neq \emptyset \}
\end{equation}
and observe that $\Theta(\Pi(t)) \subset \Gamma(t,\epsilon)$.
To relate $\#\Pi(t)$ and $\#\Gamma(t,\epsilon)$, we need to control the multiplicity of $\Theta$. 

In some instances, such as when $M$ is an oriented surface, it can be shown that $\Theta$ is injective.
In general, injectivity may fail: for example, if $\beta$ acts as translation once around a M\"obius strip whose central circle has length $t/2$, then the boundary circle of length $t$ will correspond to two elements of $\Pi(t)$, and $\Theta$ maps each of these to $\beta^2$.

\begin{definition}\label{def:dg}
Given $\gamma\in \Gamma$, let $d = d(\gamma)\in \N$ be maximal such that $\gamma = \beta^d$ for some $\beta\in\Gamma$.
\end{definition}

\begin{lemma}\label{lem:kind-of-injective}
For every $t$ and $\gamma$, we have $\#\Theta^{-1}(\gamma) \leq d(\gamma)$.
\end{lemma}
\begin{proof}
Suppose $\hv,\hw\in \Pi(t)$ have $\Theta(\hv) = \Theta(\hw) =: \gamma$. Then $\tc_{\tv}$ and $\tc_{\tw}$ are both axes of $\gamma$. By Lemma \ref{lem:Preissman}, the set of isometries fixing the endpoints of $\tc_{\tv}$ and $\tc_{\tw}$ is an infinite cyclic group; let $\beta$ be a generator of this group such that $\gamma = \beta^{d(\gamma)}$.

Lemma \ref{lem:2-axes} implies that $\hc_{\hv}$ and $\hc_{\hw}$ are free-homotopic as curves of length $|\gamma|$.
Since the elements of $C(t)$ are pairwise non-free-homotopic, we conclude that $\hc_{\hv}$ and $\hc_{\hw}$ are reparametrizations of the same closed geodesic; in other words, there is $s\in (0,|\gamma|]$ such that $\phi^s(\hv) = \hw$.
Lifting gives $\tau\in\Gamma$ such that $\phi^s\tv = \tau_* \tw$, and thus
\[
\tau \tw^+ = \lim_{r\to\infty} \tau \pi \phi^r \tw = 
\lim_{r\to\infty} \pi \phi^{r+s} \tv = \tv^+ = \tw^+.
\]
Similarly, $\tau$ fixes $\tw^-$, so $\tau = \beta^k$ for some $k\geq 1$ (it must be positive since $s>0$), and moreover $k\leq d(\gamma)$ because
\begin{align*}
k|\beta| &= |\tau| = d(\pi(\phi^s(\tv)),\pi \tw)  \\
&\leq d(\pi(\phi^s(\tv)), \pi \tv) + d(\pi \tv, \pi \tw) \leq d(\gamma) |\beta| + \tfrac \epsilon2
\leq (d(\gamma) + \tfrac 12) |\beta|,
\end{align*}
where the last two inequalities use Lemma \ref{lem:inj0} and the fact that $\epsilon \leq \frac 14 \inj(M) < |\beta|$ (by Lemma \ref{lem:ax}).
Fixing $\hv$ we see that $\hw$ is uniquely determined by $\tau$, and since there are at most $d(\gamma)$ choices for $\tau$, this proves the lemma.
\end{proof}

From Lemma \ref{lem:kind-of-injective}, we deduce that
\begin{equation}\label{eqn:upper-Pi}
\#\Pi(t) \leq
\#\Gamma(t,\epsilon) + \sum_{\substack{\gamma\in \Theta(\Pi(t)) \\ d(\gamma)\geq 2}} d(\gamma).
\end{equation}
To use \eqref{eqn:Ct-nut-1} and \eqref{eqn:upper-Pi} to estimate $\#C(t)$ in terms of $\#\Gamma(t,\epsilon)$, we need to estimate how many $\gamma\in \Theta(\Pi(t))$ have larger values of $d(\gamma)$. 

\begin{definition}
Given disjoint closed sets $\mathbf{Q},\mathbf{R} \subset\ideal$, consider for each $d\geq 2$ and $t>0$ the set
\begin{multline}\label{eqn:G0}
\Gamma_d(\mathbf{Q},\mathbf{R},t)
:= \{\gamma\in \Gamma : \xi_\gamma^- \in \mathbf{Q},\ \xi_\gamma^+ \in \mathbf{R}, \\ d(\gamma)\geq d, \text{ and } |\gamma| \in (t-\epsilon,t] \},
\end{multline}
where $\xi_\gamma^\pm$ are as in Definition \ref{def:gamma-pm}.
\end{definition}

Observe that $\{\gamma \in \Theta(\Pi(t)) : d(\gamma) \geq 2\} \subset \Gamma_2(\Pa,\Fu,t)$.

\begin{lemma}\label{lem:some-mult-small}
Given any disjoint closed sets $\mathbf{Q},\mathbf{R}\subset \ideal$, there is $K>0$ such that for all $\epsilon \in (0,\inj M]$ and $t>0$ we have
\[
\sum_{\gamma \in \Gamma_2(\mathbf{Q},\mathbf{R},t)} d(\gamma) \leq K e^{\frac 23 ht}.
\]
\end{lemma}
\begin{proof}
Let $A := \pi H^{-1}(\mathbf{Q}\times\mathbf{R}\times\{0\})$; note that $\diam A <\infty$ by Lemma \ref{lem:cpt}. Fix $q\in A$.
Freire and Ma\~n\'e proved in \cite{FM} that $h = \lim_{r\to\infty} \frac 1r \log \vol(B(q,r))$, where $B(q,r)$ is the ball of radius $r$ centered at $q$ in $X$. Thus fixing $\delta \in (0,h/3)$, there is $K_0>0$ such that
\begin{equation}\label{eqn:Freire-Mane}
\vol(B(q,r)) \leq K_0 e^{(h+\delta)r} \text{ for all } r>0.
\end{equation}
Given $d\geq 2$ and $\gamma\in \Gamma_d(\mathbf{Q},\mathbf{R},t)$, write $\gamma = \beta(\gamma)^{d(\gamma)}$.
Every axis of $\beta(\gamma)$ is an axis of $\gamma$, and thus has endpoints $\xi_\gamma^\pm$ by Lemma \ref{lem:endfix}.  In particular, we can choose $\tv\in H^{-1}(\mathbf{Q}\times\mathbf{R}\times\{0\})$ tangent to such an axis, and observe that $\phi^{|\beta(\gamma)|} \tv = \beta(\gamma)_*\tv$, so that
\begin{equation}\label{eqn:dpbp}
\begin{aligned}
d(q,\beta(\gamma) q) &\leq d(q,\pi\tv) + d(\pi\tv,\pi\beta(\gamma)_*\tv) + d(\pi\beta(\gamma)_*\tv,\beta(\gamma) q)\\
&\leq 2d(q,\pi\tv) + d(\pi v, \pi\phi^{|\beta(\gamma)|} v) \\
&\leq 2\diam A + |\beta(\gamma)|.
\end{aligned}
\end{equation}
Observe that as $\gamma$ ranges over $\Gamma_d(\mathbf{Q}, \mathbf{R},t)$,
the sets $B(\beta(\gamma)q,\inj M)$ all have the same volume $V$ (since deck transformations act isometrically) and are all disjoint because $\gamma\mapsto\beta(\gamma)$ is injective (this uses our choice of $\epsilon$ and the fact that $|\beta| \geq 2\inj(M) > \epsilon$ for all $\beta\neq\id$). 
Thus their union has volume $V \cdot \# \Gamma_d(\mathbf{Q},\mathbf{R},t)$. Since this union is contained in $B(q,2\diam A + \tfrac td)$, we can use \eqref{eqn:Freire-Mane} to deduce that
\[
V \cdot \# \Gamma_d(\mathbf{Q},\mathbf{R},t)\leq \vol(B(q,2\diam A + \tfrac td))
\leq K_0 e^{(h+\delta)(\diam A + \frac td)}.
\]
Writing $K_1 = V^{-1} K_0 e^{(h+\delta)\diam A}$, we have
\begin{equation}\label{eqn:Gd}
\# \Gamma_d(\mathbf{Q},\mathbf{R},t)
\leq K_1 e^{(h+\delta) t/d}.
\end{equation}
Observe that $d(\gamma) = |\gamma|/|\beta(\gamma)| \leq |\gamma|/\inj(M)$ by Lemma \ref{lem:ax}, so
\[
\sum_{\gamma\in \Gamma_2(\mathbf{Q},\mathbf{R},t)} d(\gamma)
= 2\#\Gamma_2(\mathbf{Q},\mathbf{R},t) +
\sum_{3\leq d \leq t/\inj(M)} \Gamma_d(\mathbf{Q},\mathbf{R},t).
\]
From \eqref{eqn:Gd} we see that $\#\Gamma_d(\mathbf{Q},\mathbf{R},t) \leq K_1 e^{(h+\delta)t/2}$ for each $d\geq 2$, and thus
\[
\sum_{\gamma\in \Gamma_2(\mathbf{Q},\mathbf{R},t)} d(\gamma) \leq 
\frac{tK_1}{\inj(M)} e^{(h+\delta) t/2}.
\]
Since $(h+\delta)/2 < 2h/3$, this proves the lemma.
\end{proof}

Combining Lemma \ref{lem:some-mult-small} with \eqref{eqn:Ct-nut-1} and \eqref{eqn:upper-Pi} gives
\begin{equation}\label{eqn:upper-P}
\#C(t)
\leq \frac{\epsilon}{t\nu_t(\hB^\epsilon)}
\big( \#\Gamma(t,\epsilon) + K e^{\frac 23 ht} \big).
\end{equation}

\begin{remark}\label{rmk:apathy}
Once we show that $\lim_{t\to\infty} \frac 1t \log \#\Gamma(t,\epsilon) = h$, it will follow that the last term in \eqref{eqn:upper-P} does not affect the asymptotics of $\#C(t)$. This same argument shows that the validity of \eqref{eqn:margulis} is not affected by whether we count all homotopy classes or only primitive ones; the number of nonprimitive ones is of a lower exponential order, and thus its contribution to the ratio of interest vanishes in the limit.
\end{remark}

\subsection{A type of closing lemma}

We will eventually complement the upper bound in \eqref{eqn:upper-P} by getting a lower bound for $\#C(t)$ in \S\ref{sec:lower-bound}, but first we need a way to guarantee that the intersection $\tS\cap \phi^{-t} \gamma_* \tB_\theta^\alpha$ actually contains a periodic orbit.

Recall from Lemma \ref{lem:ax} that every $\gamma\in \Gamma\setminus\id$ has an axis $\tc$, which projects to a closed geodesic on $M$. To produce a set of $\gamma$ for which this axis passes through $\tB$, we need a condition guaranteeing that $\tc(-\infty) \in \Pa$ and $\tc(+\infty) \in \Fu$, since then the tangents to $\tc$ will pass through $\tB$. To this end we consider
\begin{equation}\label{eqn:G*}
\Gamma^* = \Gamma^*_\theta := \{\gamma\in\Gamma :
\gamma \Fu_\theta \subset \Fu_\theta \text{ and } \gamma^{-1} \Pa_\theta \subset \Pa_\theta \}.
\end{equation}

\begin{lemma}\label{lem:closing-0}
Given any $\gamma\in \Gamma^*$, there exists an axis $\tc$ for $\gamma$ such that $\tc(-\infty) \in \Pa$ and $\tc(+\infty) \in \Fu$.
\end{lemma}
\begin{proof}
By the Brouwer fixed point theorem, $\gamma$ has one fixed point in $\Pa$ and one in $\Fu$. By Lemma \ref{lem:ax}, it has an axis $\tc$. By Lemma \ref{lem:endfix}, the fixed points of $\gamma$ are the endpoints of $\tc$.
\end{proof}

This lemma guarantees that given any $\gamma\in \Gamma^*$, there is $\tv\in \tS$ such that $\phi^{|\gamma|} \tv = \gamma_* \tv$.
We will apply this in the situation when $\tS \cap \gamma_* \phi^{-t} \tB\neq\emptyset$; in
\S\ref{sec:depths} we explore the relationship between $t$ and $|\gamma|$. The result here, together with the control on the period established there, can be thought of as a type of closing lemma.

For the moment we establish a relationship between $\Gamma_\rho(t,\alpha)$ and $\Gamma^*_\theta$
that will be important in \S\ref{sec:asymp-gam}, when we combine the lower and upper bounds into a single asymptotic estimate. 

\begin{lemma}\label{lem:c}
For every $0<\rho<\theta$, there exists $t_0 >0$ such that for all $t \geq t_0$  and $\alpha \in (0,\frac 32 \epsilon]$,
we have $\Gamma_\rho(t,\alpha) \subset \Gamma_\theta^*$.
\end{lemma}

To prove Lemma \ref{lem:c} we need the following consequence of the uniform visibility axiom. 

\begin{lemma}\label{lem:vis}
Let $\mathbf{Q},\mathbf{R}\subset \ideal$ be disjoint compact sets.
Fix a compact subset $A\subset X$ and consider for each $T>0$ the set
\[
C_\mathbf{Q}^T := \{ \tc_{p,\xi}(t) : p\in A, \xi\in \mathbf{Q}, t\geq T\},
\]
where $c_{p,\xi}$ is as in Lemma \ref{lem:p-to-xi}, and define $C_\mathbf{R}^T$ similarly.

Then given any open sets $\mathbf{U},\mathbf{V}\subset \ideal$ such that $\mathbf{U}\supset \mathbf{Q}$ and $\mathbf{V}\supset \mathbf{R}$, there exists $T>0$ such that if $\gamma\in \Gamma$ satisfies $\gamma(C_\mathbf{Q}^T) \cap A\neq\emptyset$ and $\gamma(A) \cap C_\mathbf{R}^T \neq\emptyset$, then $\gamma(\ideal\setminus \mathbf{U}) \subset \mathbf{V}$ and $\gamma^{-1}(\ideal\setminus \mathbf{V}) \subset \mathbf{U}$.
\end{lemma}

\begin{proof}
We prove the first inclusion; the second is similar. Choose $\epsilon>0$ sufficiently small that for every $p\in A$, $\xi\in \mathbf{R}$, and $\eta\in \ideal$ with $\measuredangle_p(\xi,\eta) < 2\epsilon$, we have $\eta\in \mathbf{V}$, and similarly with $\mathbf{R},\mathbf{V}$ replaced by $\mathbf{Q},\mathbf{U}$. (Existence of such an $\epsilon$ for each individual $p$ is immediate from the definition of the topology on $\ideal$; the fact that $\epsilon$ can be chosen independently of $p$ uses compactness of $A$ and continuity of the map $p\mapsto \dot{\tc}_{p,\xi}(0)$ from Lemma \ref{lem:p-to-xi}.) Let $L=L(\epsilon)$ be given by the uniform visibility property (see Definition \ref{def:vis}),
and let $T = 2L + \diam A$.

Fix $\gamma$ satisfying the hypothesis of the lemma, and let $x,y\in A$ be such that $\gamma y\in C_\mathbf{R}^T$ and $\gamma^{-1} x \in C_\mathbf{Q}^T$. By definition of $C_\mathbf{R}^T$, there is $p\in A$ such that $\xi := \tc_{p,\gamma y}(\infty) \in \mathbf{R}$ and $d(p,\gamma y) \geq T$. Similarly, there is $q\in A$ such that $\tc_{q,\gamma^{-1} x}(\infty) \in \mathbf{Q}$ and $d(q,\gamma^{-1} x) \geq T$.

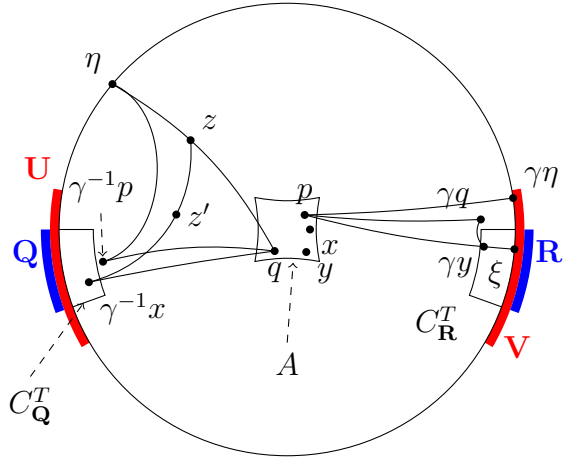
\begin{figure}[htbp]
\begin{tikzpicture}[scale=3]
\fill[red] (170:1) node[above left]{$\mathbf{U}$}
arc (170:210:1) -- (210:1.04) arc (210:170:1.04) -- cycle;
\fill[blue] (180:1.04) node[below left]{$\mathbf{Q}$}
arc (180:200:1.04) -- (200:1.08) arc (200:180:1.08) -- cycle;
\fill[red] (10:1) arc (10:-30:1) -- (-30:1.04) 
node[right]{$\mathbf{V}$} arc (-30:10:1.04) -- cycle;
\fill[blue] (0:1.04) node[below right]{$\mathbf{R}$}
arc (0:-20:1.04) -- (-20:1.08) arc (-20:0:1.08) -- cycle;
\draw (45:.2) to[bend left=10] (135:.2) to[bend left=10] (225:.2) to[bend left=10] (315:.2) to[bend left=10] cycle;
\draw[dashed,->] (-90:.5) node[below]{$A$} -- (-80:.15);
\draw (200:1) -- (200:.85) arc (200:180:.85) -- (180:1);
\draw[dashed,->] (-150:1.3) node[below]{$C_\mathbf{Q}^T$} -- (200:.95);
\draw (-20:1) -- (-20:.85) node[below left]{$C_\mathbf{R}^T$} arc (-20:0:.85) -- (0:1);
\draw (0,0) circle (1);
\coordinate (q) at (-120:.11);
\fill (q) node[below]{$q$} circle(.5pt);
\coordinate (g-x) at (-165:.9);
\fill (g-x) node[below right]{$\gamma^{-1} x$} circle(.5pt);
\draw (q) to[bend right=2] (g-x);
\coordinate (p) at (40:.1);
\fill (p) node[above]{$p$} circle(.5pt);
\coordinate (xi) at (-5:1);
\fill (xi) node[below left]{$\xi$} circle(.5pt);
\draw (p) to[bend right=5] coordinate[pos=0.87] (gy) (xi);
\fill (gy) node[below left]{$\gamma y$} circle(.5pt);
\fill (0:.1) node[below right]{$x$} circle(.5pt);
\fill (-50:.13) node[below right]{$y$} circle(.5pt);
\coordinate (gq) at (3:.85);
\fill (gq) circle(.5pt);
\draw (gy) to[bend left=40] (gq) node[above left]{$\gamma q$};
\draw (p) to[bend right=4] (gq);
\coordinate (ge) at (8:1);
\fill (ge) node[above right]{$\gamma\eta$} circle(.5pt);
\draw (p) to[bend right=3] (ge);
\coordinate (e) at (140:1);
\fill (e) node[above left]{$\eta$} circle(.5pt);
\draw (e) to[bend left=15] coordinate[pos=.4](z) (q);
\fill (z) node[above right]{$z$} circle(.5pt);
\coordinate (g-p) at (-170:.82);
\fill (g-p) circle(.5pt);
\draw[dashed,->,shorten >=2pt] (175:.82) node[above]{$\gamma^{-1}p$} -- (g-p);
\draw (g-p) to[bend right=70] (e);
\draw (g-p) to[bend left=10] (q);
\draw (g-x) to[bend right=40] coordinate[pos=.6](z') (z);
\fill (z') node[right]{$z'$} circle(.5pt);
\end{tikzpicture}
\caption{Proving Lemma \ref{lem:vis}}
\label{fig:vis}
\end{figure}

Given an arbitrary $\eta\in \ideal \setminus \mathbf{U}$, we will show that $\measuredangle_p(\xi,\gamma\eta) < 2\epsilon$. Observe that
\begin{equation}\label{eqn:angle-bd}
\measuredangle_p(\xi,\gamma\eta)
\leq \measuredangle_p(\xi,\gamma q) + \measuredangle_p(\gamma \eta,\gamma q) 
= \measuredangle_p(\gamma y, \gamma q) + \measuredangle_{\gamma^{-1} p}(\eta, q).
\end{equation}
We bound $\measuredangle_p(\gamma y, \gamma q)$ by observing that $d(\gamma y, \gamma q) = d(y,q) \leq \diam A$, while $d(p,\gamma y) \geq T$. Thus for every point $y'$ on the geodesic segment connecting $\gamma y$ and $\gamma q$, we have $d(p,y') \geq T - \diam A \geq L$, and the uniform visibility property implies that
\begin{equation}\label{eqn:angle-bd-1}
\measuredangle_p(\gamma y, \gamma q) < \epsilon.
\end{equation}
To bound $\measuredangle_{\gamma^{-1} p}(\eta, q)$, first note by our choice of $\epsilon$ and $q$, for every $z$ on the geodesic ray from $q$ to $\eta$
we have  $\measuredangle_q(\gamma^{-1} x, z) = \measuredangle_q(\gamma^{-1} x, \eta) \geq 2\epsilon$.  By the uniform visibility property, for each such $z$ there is $z'$ on the geodesic segment from $z$ to $\gamma^{-1} x$ such that $d(q,z') \leq L$. Note also that
\[
d(z,\gamma^{-1} x) \geq d(z',\gamma^{-1} x) \geq d(q,\gamma^{-1} x) - d(q,z')
\geq T-L,
\]
and thus
\[
d(z,\gamma^{-1} p) \geq d(z,\gamma^{-1} x) - d(\gamma^{-1} x, \gamma^{-1} p)
\geq T - L - \diam A \geq L
\]
since $\gamma^{-1}$ acts isometrically. Applying the uniform visibility property once more gives $\measuredangle_{\gamma^{-1} p}(\eta,q) < \epsilon$. Together with \eqref{eqn:angle-bd} and \eqref{eqn:angle-bd-1} this gives $\measuredangle_p(\xi,\gamma\eta) < 2\epsilon$, and thus $\gamma\eta \in \mathbf{V}$, which proves the lemma.
\end{proof}

\begin{proof}[Proof of Lemma \ref{lem:c}]
It suffices to consider the case $\alpha = \frac 32\epsilon$.
We apply Lemma \ref{lem:vis} by putting $\mathbf{Q}=\Pa_\rho$, $\mathbf{R}=\Fu_\rho$, $A=\pi \tB_\rho^\alpha$, and letting $\mathbf{U},\mathbf{V}$ be the interiors of $\Pa_\theta,\Fu_\theta$, respectively. Let $T$ be given by that lemma; then for every $t\geq T$ and $\gamma\in \Gamma_\rho(t,\alpha)$, there exists $v\in \tS_\rho \subset \tB_\rho^\alpha$ such that $w := \gamma_*^{-1} \phi^{t} v \in \tB_\rho^\alpha$.
Putting $x = \pi v \in A$ we have
\[
\gamma^{-1} x = \pi \gamma_*^{-1} v = \pi \phi^{-t} w;
\]
since $\pi w \in A$, $w^- \in \Pa_\rho = \mathbf{Q}$, and $t\geq T$, this implies that $\gamma^{-1} x \in C_\mathbf{Q}^T$. Similarly, putting $y = \pi w \in A$ gives
\[
\gamma y = \pi \gamma_* w = \pi \phi^t v;
\]
since $\pi v \in A$, $v^+ \in \Fu_\rho = \mathbf{R}$, and $t\geq T$, this gives $\gamma y \in C_\mathbf{R}^T$, and Lemma \ref{lem:vis} yields the desired result.
\end{proof}

\subsection{Depths of intersections}\label{sec:depths}

Recall that we fix $\tv_0\in \mathcal{E}$ and $p = \pi\tv_0$.
Consider the set
\begin{equation}\label{eqn:G*t}
\Gamma_\theta^*(t,\alpha) := \Gamma^* \cap \Gamma_\theta(t,\alpha)
= \{\gamma\in \Gamma^* : \tS_\theta \cap \gamma_* \phi^{-t} \tB_\theta^\alpha \neq\emptyset \}.
\end{equation}
We need to understand how $|\gamma|$ and $t$ are related when $\gamma\in \Gamma^*(t,\alpha)$.

\begin{definition}\label{def:bxg}
Given $\xi\in \ideal$ and $\gamma\in\Gamma$, let $b_\xi^\gamma := b_\xi(\gamma p, p)$. 
\end{definition}

\begin{lemma}\label{lem:b-close}
For all $\xi,\eta\in \Pa$ and every $\gamma \in \Gamma_\theta(t,\alpha)$ with $t>0$, we have $|b_\xi^\gamma - b_\eta^\gamma| < \epsilon^2$.
\end{lemma}
\begin{proof}
Given any such $\gamma$, there exists $v\in \tS \cap \gamma_* \phi^{-t} \tB^\alpha$,
so $\phi^t v \in \gamma_* \tB^\alpha$. Thus there is $q\in \pi B^\alpha$ such that $\gamma q = \pi \phi^t v \in \pi H^{-1}(\Pa_\theta\times\Fu_\theta\times [0,\infty))$. Since $p \in \pi B^\alpha$, we have $d(\gamma p, \gamma q) = d(p,q) < 2\epsilon$ by Lemma \ref{lem:inj}, so $\gamma p$ satisfies the condition of Lemma \ref{lem:eps'}, and we have $|b_\xi(\gamma p, p) - b_\eta(\gamma p, p)| < \epsilon^2$ for all $\xi,\eta \in \Pa_\theta$.
\end{proof}

\begin{lemma}\label{lem:b-period}
If $\tc$ is an axis of $\gamma$ and $\xi=\tc(-\infty)$, then $b_\xi^\gamma = |\gamma|$.
\end{lemma}
\begin{proof}
Let $q$ be any point on $\tc$; then $|\gamma| = d(\gamma q,q) = b_\xi(\gamma q,q)$. Thus \eqref{eqn:pqr} gives
\begin{align*}
|\gamma| - b_\xi^\gamma
&= b_\xi(\gamma q,q) - b_\xi(\gamma p,p) \\
&= (b_\xi(\gamma q, \gamma p) + b_\xi(\gamma p,q))
- (b_\xi(\gamma p, q) + b_\xi(q,p)) \\
&= b_\xi(\gamma q,\gamma p) - b_{\gamma\xi}(\gamma q,\gamma p),
\end{align*}
where the last step uses the fact that $\gamma$ acts isometrically so Busemann functions are unchanged when $\gamma$ is applied to all three arguments. Finally, $\gamma\xi=\xi$, so the lemma is proved.
\end{proof}

\begin{lemma}\label{lem:s-interval}
Given any $\gamma\in\Gamma^*$ and any $t\in\R$, we have\footnote{The asymmetry in this lemma -- the fact that $\tw^-$ appears and $\tw^+$ does not -- is due to the fact that we define $s(\tw)$ using $\tw^-$; in the original proof by Margulis, one must similarly choose whether to construct the front of the flow box as a union of stable leaves or unstable leaves. Of course either choice leads to a similar argument.}
\[
\tS \cap \phi^{-t} \gamma_* \tB^\alpha
= \{ \tw\in E^{-1}(\Pa \times \gamma \Fu) :
s(\tw) \in [0,\epsilon^2] \cap (b_{\tw^-}^\gamma - t + [0,\alpha])\}.
\]
\end{lemma}
\begin{proof}
To prove that $\tS\cap \phi^{-1}\gamma_*\tB^\alpha\subset E^{-1}(\Pa\times\gamma\Fu)$, we observe that if $E(\tw) \notin \Pa\times \gamma\Fu$, then either $\tw^-\notin \Pa$, so $\tw\notin \tS$, or $\tw^+\notin \gamma\Fu$, so 
$\tw\notin \phi^{-t}\gamma_*\tB^\alpha$.

It remains to show that given $\tw\in E^{-1}(\Pa\times\gamma\Fu)$, we have
\begin{align}
\label{eqn:S-int}
\tw\in \tS \ &\Leftrightarrow\ s(\tw) \in [0,\epsilon^2], \text{ and}\\
\label{eqn:gB-int}
\tw\in \phi^{-t} \gamma_* \tB^\alpha
\ &\Leftrightarrow\ s(\tw) \in b_{\tw^-}^\gamma - t + [0,\alpha].
\end{align}
The first of these is immediate from the definition of $\tS$. For the second, we observe that $s(\tv) = b_{\tv^-}(\pi \tv,p) = b_{\gamma \tv^-}(\gamma\pi \tv,\gamma p)$, and thus
\begin{align*}
\gamma_* \tB^\alpha
&= \{\gamma_* \tv : \tv\in E^{-1}(\Pa\times\Fu)
\text{ and } b_{\tv^-}(\pi \tv,p) \in [0,\alpha] \} \\
&= \{ \tw \in E^{-1}(\gamma\Pa\times\gamma\Fu) : 
b_{\tw^-}(\pi \tw,\gamma p) \in [0,\alpha]\}
\end{align*}
By \eqref{eqn:pqr} and \eqref{eqn:pq}, we have
\[
b_{\tw^-}(\pi \tw, \gamma p)
= b_{\tw^-}(\pi \tw, p) + b_{\tw^-}(p,\gamma p)
= s(\tw) - b_{\tw^-}^\gamma;
\]
moreover, since $s(\phi^t \tw) = s(\tw) + t$ by \eqref{eqn:+t}, we see that $\phi^t \tw \in \gamma_*\tB^\alpha$ if and only if $s(\tw) - b_{\tw^-}^\gamma + t \in [0,\alpha]$, which proves \eqref{eqn:gB-int} and completes the proof of the lemma.
\end{proof}

\begin{lemma}\label{lem:|g|-to-t}
If $\gamma\in \Gamma^*(t,\alpha)$,
then $|\gamma| \in [t-\alpha-\epsilon^2,t+2\epsilon^2]$.
\end{lemma}
\begin{proof}
Since $\tS \cap \phi^{-t} \gamma_* \tB^\alpha$ is nonempty,
Lemma \ref{lem:s-interval}  guarantees that there exist $\zeta\in \Pa$ and $\tau\in [0,\epsilon^2]$ such that $\tau \in b_\zeta^\gamma - t + [0,\alpha]$; this implies that
\[
b_\zeta^\gamma \in t + \tau - [0,\alpha] \subset [t-\alpha, t + \epsilon^2].
\]
From Lemmas \ref{lem:b-close} and \ref{lem:b-period}, we have
$|\gamma| \in [b_\zeta^\gamma - \epsilon^2, b_\zeta^\gamma + \epsilon^2]$, which proves the lemma.
\end{proof}

\subsection{A lower counting bound}\label{sec:lower-bound}

Now we can obtain a lower bound for $\#\Pi(t) = \#\Theta(\Pi(t))$, and hence for $\#C(t)$. 
We define
\begin{equation}\label{eqn:G'}
\Gamma'(t,\alpha) := \{\gamma\in\Gamma^*(t,\alpha) :
\gamma\neq \beta^n \text{ for any } \beta\in \Gamma, n\geq 2\}.
\end{equation}

\begin{lemma}\label{lem:lower-bound}
Let $\alpha = \epsilon - 4\epsilon^2$; then $\Theta(\Pi(t)) \supset \Gamma'(t-2\epsilon^2,\alpha)$, and from \eqref{eqn:Ct-nut-1} we obtain
\begin{equation}\label{eqn:lower-P-a}
\#C(t)
\geq \frac{\alpha}t \cdot \frac{\#\Gamma'(t-2\epsilon^2,\alpha)}{\nu_t(\hB^\alpha)}.
\end{equation}
\end{lemma}
\begin{proof}
By Lemma \ref{lem:closing-0}, given any $\gamma\in \Gamma'(t-2\epsilon^2,\alpha)$, there exists $\tv\in H^{-1}(\Pa\times\Fu\times\{0\})$ such that $\phi^{|\gamma|}\tv = \gamma_* \tv$. By Lemma \ref{lem:|g|-to-t}, we have
\begin{align*}
|\gamma| &\geq (t-2\epsilon^2) - \alpha - \epsilon^2
= t-3\epsilon^2 - (\epsilon - 4\epsilon^2) > t-\epsilon, \\
|\gamma| &\leq (t-2\epsilon^2) + 2\epsilon^2 = t,
\end{align*}
and thus $\hv := {\pr_*}\tv \in SM$ lies on a closed geodesic $\hc_{\hv}$ with length $|\gamma| \in (t-\epsilon,t]$ (this uses the assumption that $\gamma$ is not a nontrivial power of another isometry).
By Lemmas \ref{lem:hom-lift} and \ref{lem:2-axes}, every closed geodesic in the free homotopy class of $\hc_{\hv}$ has the same period, so we choose $\hc\in C(t)$ that is homotopic to $\hc_{\hv}$, and note from Lemma \ref{lem:hom-lift} that there is a lift $\tc$ with $\tw := \dot{\tc}(0) \in H^{-1}(\Pa\times\Fu\times\{0\})$. It follows from \eqref{eqn:Pit} that $\hw := {\pr_*} \tw \in \Pi(t)$, and from Definition \ref{def:Theta} and irreducibility of $\gamma$ that $\Theta(w) = \gamma$,
which proves the lemma.
\end{proof}

Observe that $\Gamma^*(t-2\epsilon^2,\alpha) \setminus \Gamma'(t-2\epsilon^2,\alpha)
\subset \Gamma_2(\Pa,\Fu,t)$, so by Lemma \ref{lem:some-mult-small} we have
\[
\#\Gamma^*(t-2\epsilon^2,\alpha) - \#\Gamma'(t-2\epsilon^2,\alpha)                                                         
\leq K e^{\frac 23ht}.
\]
Using this together with Lemma \ref{lem:lower-bound} gives
\begin{equation}\label{eqn:lower-P}
\#C(t)
\geq \frac{\alpha}{t\nu_t(\hB^\alpha)} \big( \#\Gamma^*(t-2\epsilon^2,\alpha) - K e^{\frac 23 ht} \big),
\quad \alpha = \epsilon - 4\epsilon^2.
\end{equation}
From \eqref{eqn:upper-P} and \eqref{eqn:lower-P} we see that we must now estimate $\#\Gamma(t,\alpha)$ and $\#\Gamma^*(t,\alpha)$; we do this in the next section.
Then in \S\ref{sec:fin} we use these estimates to deduce the equidistribution result (Theorem \ref{thm:equidist}) and combine them with the lemmas from this section to obtain good estimates on $\#C(t)$. This yields estimates on $\#P(t)$ via Riemann sums, and sending $\epsilon\to 0$ yields integrals that we can evaluate to prove Theorem \ref{thm:main}.

\section{Consequences of scaling and mixing}\label{sec:mixing}

In what follows, it will be convenient to use the following notations, along with $f\sim g$:
\begin{align*}
f(t) = e^{\pm C} g(t) &\quad \Leftrightarrow\quad
e^{-C} g(t) \leq f(t) \leq e^C g(t) \text{ for all } t; \\
f(t) \lesssim g(t) &\quad\Leftrightarrow\quad
\limsup_{t\to\infty} \frac{f(t)}{g(t)} \leq 1; \\ 
f(t) \gtrsim g(t) &\quad\Leftrightarrow\quad
\liminf_{t\to\infty} \frac{f(t)}{g(t)} \geq 1; \\
f(t) \sim e^{\pm C} g(t) &\quad\Leftrightarrow\quad
e^{-C} g(t) \lesssim f(t) \lesssim e^C g(t).
\end{align*}

\subsection{Scaling}\label{sec:scaling}

Given $\alpha \leq \frac 32\epsilon$,
we will estimate $\#\Gamma(t,\alpha)$ and $\#\Gamma^*(t,\alpha)$ using the product structure, scaling properties, and mixing property of $\hm$. 
Note from Lemma \ref{lem:s-interval} that although $\tB$ and $\tS$ have a product structure given by \eqref{eqn:B} and \eqref{eqn:S}, the sets $\tS \cap \phi^{-t}\gamma_* \tB^\alpha$ do not always have such a structure. Using the formula in that lemma, though, we can give a sufficient condition for these intersections to have a product structure.

\begin{lemma}\label{lem:full-branch}
Given any $\alpha,t>0$ and $\gamma\in \Gamma^*(t,\alpha)$, we have
\[
\tS\cap \phi^{-(t+2\epsilon^2)}\gamma_* \tB^{\alpha+4\epsilon^2} = H^{-1}(\Pa\times\gamma\Fu\times[0,\epsilon^2])
=: \tS^\gamma.
\]
\end{lemma}
\begin{proof}
By Lemma \ref{lem:s-interval}, the fact that $\tS \cap \phi^{-t} \gamma_* \tB^\alpha \neq\emptyset$ implies existence of $\eta \in \Pa$ such that
\[
(b_\eta^\gamma - t + [0,\alpha]) \cap [0,\epsilon^2] \neq \emptyset,
\]
from which we deduce that
\[
b_\eta^\gamma - t - \epsilon^2 + [0,\alpha+2\epsilon^2] \supset [0,\epsilon^2].
\]
By Lemma \ref{lem:b-close}, it follows that every $\xi\in \Pa$ has
\[
(b_\xi^\gamma - t - \epsilon^2 + [0,\alpha+2\epsilon^2]) \cap [0,\epsilon^2] \neq\emptyset,
\]
which in turn implies that
\[
b_\xi^\gamma - t - 2\epsilon^2 + [0,\alpha+4\epsilon^2] \supset [0,\epsilon^2].
\]
By Lemma \ref{lem:s-interval}, this completes the proof.
\end{proof}

Given $\gamma\in \Gamma^*$, let $\tS^\gamma := H^{-1}(\Pa\times\gamma\Fu\times[0,\epsilon^2])$ as in Lemma \ref{lem:full-branch}, and write $\hS^\gamma = {\pr_*}\tS^\gamma \subset SM$.

\begin{lemma}\label{lem:scaling}
For each $\gamma \in \Gamma^*$, we have
\[
\tm(\tS^\gamma) = e^{\pm 2 h \epsilon} e^{-h|\gamma|} \tm(\tS),
\]
and similarly with $\tm,\tS,\tS^\gamma$ replaced by $\hm,\hS,\hS^\gamma$.
\end{lemma}
\begin{proof}
Since $\hm(\hS^\gamma) = \tm(\tS^\gamma) = \epsilon^2 \bar\mu(\Pa\times\gamma\Fu)$, and $\hm(\hS) = \tm(\tS) = \epsilon^2\bar\mu(\Pa\times\Fu)$, it suffices to show that $\bar\mu(\Pa\times\gamma\Fu) = e^{\pm 2h\epsilon} e^{-h|\gamma|} \bar\mu(\Pa\times\Fu)$.

We will use the definition of $\bar\mu$ in \eqref{eqn:barmu}
and the scaling properties of the conformal measure in Definition \ref{def:conformal}; thus we need to control $\beta_p(\xi,\eta)$ for $(\xi,\eta)\in\Pa\times\Fu$, and $b_\eta(\gamma^{-1} p,p)$ for $\eta\in\Fu$. The former will be close to $0$, and the latter will be close to $|\gamma|$.

Indeed, given $(\xi,\eta)\in\Pa\times\Fu$, we can take $q$ to lie on a geodesic connecting $\xi$ and $\eta$, with $b_\xi(q,p)=0$; then \eqref{eqn:beta-p} gives
\[
|\beta_p(\xi,\eta)| = |b_\xi(q,p) + b_\eta(q,p)|
\leq d(q,p) < \epsilon/2,
\]
where the last inequality uses Lemma \ref{lem:inj0}. Using this together with \eqref{eqn:barmu} gives
\[
\bar\mu(\Pa\times\Fu) = e^{\pm h \epsilon/2} \mu_p(\Pa) \mu_p(\Fu)
\text{ and }
\bar\mu(\Pa\times\gamma\Fu) = e^{\pm h \epsilon/2} \mu_p(\Pa) \mu_p(\gamma\Fu),
\]
and thus
\begin{equation}\label{eqn:ratio1}
\frac{\bar\mu(\Pa\times\gamma\Fu)}{\bar\mu(\Pa\times\Fu)} = e^{\pm h\epsilon} \frac{\mu_p(\gamma\Fu)}{\mu_p(\Fu)}.
\end{equation}
We will estimate the latter ratio using Definition \ref{def:conformal}, whose Property \ref{PS-b} gives
\[
\mu_p(\gamma \Fu) =  \mu_{\gamma^{-1} p}(\Fu),
\]
and whose Property \ref{PS-c} gives
\[
\frac{d\mu_{\gamma^{-1} p}}{d\mu_p}(\eta) 
= e^{-h b_\eta(\gamma^{-1} p,p)}.
\]
When $\eta = \tc(-\infty)$, where $\tc$ is the axis for $\gamma^{-1}$, we have $\eta\in \Fu$ because $\gamma\in \Gamma^*$, and $b_\eta(\gamma^{-1} p,p) = |\gamma^{-1}| = |\gamma|$ by Lemma \ref{lem:b-period}. For other choices of $\eta\in\Fu$, Lemma \ref{lem:eps'} implies that the value of $b_\eta(\gamma^{-1}p,p)$ varies by at most $\epsilon^2$. We conclude that $\mu_p(\gamma\Fu) = e^{\pm\epsilon^2} e^{-h|\gamma|} \mu_p(\Fu)$, and together with \eqref{eqn:ratio1} this proves the lemma.
\end{proof}

From Lemmas \ref{lem:|g|-to-t} and \ref{lem:scaling} we immediately deduce the following.

\begin{corollary}\label{cor:G*t-scaling}
Given $\alpha\leq \frac 32\epsilon$ and $\gamma\in\Gamma^*(t,\alpha)$, we have $\big|t-|\gamma|\big| \leq 2\epsilon$, and thus $\tm(\tS^\gamma) = e^{\pm 4h\epsilon} e^{-ht}\tm(\tS)$, and similarly for $\hm,\hS,\hS^\gamma$.
\end{corollary}

\subsection{Asymptotic estimates}\label{sec:asymp-gam}

It follows from Lemmas \ref{lem:c} and \ref{lem:full-branch} that given any $\alpha\in (0,\frac 32\epsilon]$ and $\rho\in (0,\theta)$, for all sufficiently large $t$ we have
\[
\hS_\rho \cap \phi^{-t} \hB_\rho^\alpha
\subset \bigcup_{\gamma\in \Gamma^*(t,\alpha)} \hS_\theta^\gamma
\subset \hS_\theta \cap \phi^{-(t+2\epsilon^2)} \hB_\theta^{\alpha+4\epsilon^2}.
\]
Using Corollary \ref{cor:G*t-scaling} gives
\begin{equation}\label{eqn:G*-Sg}
\hm(\hS_\theta^\gamma) = e^{\pm 4h\epsilon} e^{-ht} \hm(\hS_\theta) 
\end{equation}
for all $\gamma\in \Gamma^*(t)$, and thus
\begin{multline*}
e^{-4h\epsilon} \hm(\hS_\rho \cap \phi^{-t} \hB_\rho^\alpha)
\leq \#\Gamma^*(t,\alpha) e^{-ht}\hm(\hS_\theta) \\
\leq e^{4h\epsilon} \hm(\hS_\theta \cap \phi^{-(t+2\epsilon^2)} \hB_\theta^{\alpha+4\epsilon^2}).
\end{multline*}
Sending $t\to\infty$, using mixing, and dividing through by $\hm(\hS_\theta)\hm(\hB_\theta^\alpha) = \tm(\tS_\theta)\tm(\tB_\theta^\alpha)$, we get
\[
e^{-4h\epsilon}\frac{\tm(S_\rho)}{\tm(S_\theta)} \frac{\tm(\tB_\rho^\alpha)}{\tm(\tB_\theta^\alpha)}
\lesssim \frac{\#\Gamma^*(t,\alpha)}{e^{ht}\tm(\tB_\theta^\alpha)}
\lesssim e^{4h\epsilon} \frac{\tm(\tB_\theta^{\alpha+4\epsilon^2})}{\tm(\tB_\theta^\alpha)}
\]
By \eqref{eqn:SB-cty}, $\theta$ is a point of continuity for $\rho\mapsto \tm(\tS_\rho)$ and $\rho\mapsto \tm(\tB_\rho^\alpha)$, so we can send $\rho\nearrow \theta$ and obtain
\begin{equation}\label{eqn:counting-bounds}
e^{-4h\epsilon}
\lesssim \frac{\#\Gamma^*(t,\alpha)}{e^{ht}\tm(\tB_\theta^\alpha)}
\lesssim e^{4h\epsilon} (1+4\epsilon^2/\alpha).
\end{equation}
We will also need to use \eqref{eqn:counting-bounds} with $\Gamma^*$ replaced by $\Gamma$. Observe that for every $\rho>\theta$, Lemma \ref{lem:c} gives 
$\Gamma_\theta^*(t,\alpha) \subset \Gamma_\theta(t,\alpha)
\subset \Gamma_\rho^*(t,\alpha)$ for all sufficiently large $t$, and thus \eqref{eqn:counting-bounds} gives
\begin{multline*}
e^{-4h\epsilon} e^{ht} \tm(\tB_\theta^\alpha)
\lesssim \#\Gamma_\theta^*(t,\alpha)
\lesssim \#\Gamma_\theta(t,\alpha) \\
\lesssim \#\Gamma_\rho^*(t,\alpha)
\lesssim e^{4h\epsilon}(1+4\epsilon^2/\alpha) \tm(\tB_\rho^\alpha)
e^{ht}
\end{multline*}
Sending $\rho\searrow \theta$ and using \eqref{eqn:SB-cty} gives
\begin{equation}\label{eqn:G-counting}
e^{-4h\epsilon} e^{ht} \tm(B_\theta^\alpha)
\lesssim \#\Gamma_\theta(t,\alpha)
\lesssim e^{4h\epsilon}(1+4\epsilon^2/\alpha) e^{ht} \tm(B_\theta^\alpha).
\end{equation}

\section{Completion of the proofs}\label{sec:fin}

\subsection{Equidistribution to the measure of maximal entropy}\label{sec:equidist}

The following is a standard result in ergodic theory; see for example the proof of \cite[Theorem 9.10]{Wa}.

\begin{lemma}\label{lem:mme}
Let $Y$ be a compact metric space and $\phi^t$ a continuous flow on $Y$. Fix $\epsilon>0$ and suppose that  $E_t \subset Y$ is a $(t,\epsilon)$-separated set for all sufficiently large $t$. Then the measures $\mu_t$ defined by
\begin{equation}\label{eqn:mut}
\mu_t(A) := \frac 1{\# E_t} \sum_{v\in E_t} \frac 1t\int_0^t \chi_A(\phi^s v) \,ds
\end{equation}
have the property that if $t_k\to\infty$ and the weak* limit $\mu = \lim_{k\to\infty} \mu_{t_k}$ exists, then $h_\mu(\phi^1) \geq \limsup_{k\to\infty} \frac 1{t_k} \log\# E_{t_k}$.
\end{lemma}

In particular, if $\lim_{t\to\infty} \frac 1t \log \# E_t$ is the topological entropy of the flow, then every weak* limit point of the family of measures $\mu_t$ is a measure of maximal entropy.

\begin{corollary}\label{cor:mme}
Let $Y$ be a compact metric space and $\phi^t$ a continuous flow on $Y$ with a unique measure of maximal entropy $\mu$. Fix $\epsilon>0$. If $\{E_t\subset Y\}_{t}$ is a family of $(t,\epsilon)$-separated sets for which $\lim_{t\to\infty} \frac 1t \log \# E_t$ is the topological entropy of the flow, then the measures $\mu_t$ defined in \eqref{eqn:mut} converge to $\mu$ in the weak* topology.
\end{corollary}

\begin{proof}[Proof of Theorem \ref{thm:equidist}]
Let $C(t)$ be as in the statement of Theorem \ref{thm:equidist}. 
By \eqref{eqn:lower-P}, we have
\[
\#C(t) \geq \frac\alpha t \big(\#\Gamma^*(t-2\epsilon^2,\alpha) - K e^{\frac 23 ht}\big),
\quad \alpha = \epsilon - 4\epsilon^2,
\]
and then \eqref{eqn:counting-bounds} gives
\begin{equation}\label{eqn:llim-h}
\varliminf_{t\to\infty} \frac 1t \log \#C(t) \geq h.
\end{equation}
We claim that the set $\{ \dot\hc(0) : \hc\in C(t)\}$ is $(t,\epsilon)$-separated for any $\epsilon \in (0,\inj(M))$. Indeed, if it were not, then $C(t)$ would contain $\hc_0\neq \hc_1$ such that $d(\hc_0(t),\hc_1(t)) \leq \inj(M)$ for all $t\in\R$, and thus there would be lifts $\tc_0,\tc_1$ satisfying the same inequality, so that $\tc_0(\pm\infty)=\tc_1(\pm\infty)$. By Lemma \ref{lem:Preissman} there is $\gamma\in\Gamma$ such that the axial isometries of $\tc_0,\tc_1$ are both of the form $\gamma^k$. Since $\epsilon < \inj(M) < |\gamma|$, the value of $k$ is the same for both geodesics, so Lemma \ref{lem:2-axes} implies that $\hc_0$ and $\hc_1$ are homotopic, contradicting our assumption.

We conclude that the limit in \eqref{eqn:llim-h} exists and is equal to $h$, and Corollary \ref{cor:mme} implies that the measures $\nu_t$ from \eqref{eqn:nut} converge to the unique measure of maximal entropy, which completes the proof of Theorem \ref{thm:equidist}.
\end{proof}

By Theorem \ref{thm:equidist} and \eqref{eqn:bdry-null}, we have $\nu_t(\hB^\epsilon) \to \hm(\hB^\epsilon) = \tm(\tB^\epsilon)$. Now \eqref{eqn:upper-P} and \eqref{eqn:G-counting} give
\begin{equation}\label{eqn:P-upper}
\#C(t) \lesssim \frac \epsilon t \cdot \frac{\#\Gamma(t,\epsilon) + K e^{\frac 23 ht}}{m(B^\epsilon)}
\lesssim e^{4h\epsilon}(1+4\epsilon)\frac\epsilon t e^{ht}.
\end{equation}
Similarly, \eqref{eqn:lower-P} and \eqref{eqn:counting-bounds} (with $\alpha = \epsilon - 4\epsilon^2$) give
\[
\#C(t) \gtrsim \frac {\epsilon-4\epsilon^2} t \cdot \frac{\#\Gamma^*(t-2\epsilon^2,\alpha) - K e^{\frac 23 ht}}{m(B^\alpha)}
\gtrsim (1-4\epsilon) e^{-4h\epsilon} \frac \epsilon te^{ht} e^{-2h\epsilon^2}.
\]
Combining this with 
\eqref{eqn:P-upper} gives
\begin{equation}\label{eqn:P-asymp}
\#C(t) \sim e^{\pm Q\epsilon} \frac \epsilon te^{ht}
\end{equation}
where $Q$ is a universal constant depending only on $h$.

\subsection{Riemann sums and integrals}\label{sec:completion}

To complete the proof, we use \eqref{eqn:P-asymp} to estimate $\#P(t)$ via a Riemann sum, and then send $\epsilon\to 0$ to convert this to an integral and send the size of the error to $0$. See \cite[\S13]{Ri} for a more axiomatic approach to this step.

Since $\#P(T) \to \infty$ as $T\to\infty$, and each $\#P(b)$ is finite, we see that $\#P(T) \sim \#(P(T) \setminus P(b))$ for every $b>0$.
Summing over $t_k = T - k\epsilon$ gives
\begin{equation}\label{eqn:R-sum}
\sum_{k=0}^{\lfloor(T-b)/\epsilon\rfloor}
\#C(t_k) 
\leq \#(P(T) \setminus P(b))
\leq \sum_{k=0}^{\lceil(T-b)/\epsilon\rceil}
\#C(t_k).
\end{equation}
By \eqref{eqn:P-asymp} we can choose $b$ sufficiently large that $\#C(t) = e^{\pm2Q\epsilon} \frac \epsilon t e^{ht}$ for all $t\geq b$, and thus \eqref{eqn:R-sum} gives
\begin{equation}\label{eqn:R-sum-2}
e^{-2Q\epsilon} 
\sum_{k=0}^{\lfloor(T-b)/\epsilon\rfloor}
\epsilon \frac{e^{ht_k}}{t_k}
\leq \#(P(T) \setminus P(b))
\leq 
e^{2Q\epsilon} \sum_{k=0}^{\lceil(T-b)/\epsilon\rceil}
\epsilon \frac{e^{ht_k}}{t_k}.
\end{equation}
Assume that $b$ is also chosen large enough that $t\mapsto \frac{e^{ht}}t$ is nondecreasing on $(b,\infty)$. Then the first sum in \eqref{eqn:R-sum-2} is an upper bound for $\int_b^T \frac{e^{ht}}t\,dt$, and the second sum is a lower bound for 
$\int_b^{T+\epsilon} \frac{e^{ht}}t\,dt$, so we conclude that
\begin{equation}\label{eqn:int-bounds}
e^{-2Q\epsilon} \int_b^T \frac{e^{ht}}t\,dt
\leq \#(P(T) \setminus P(b))
\leq e^{2Q\epsilon} \int_b^{T+\epsilon} \frac{e^{ht}}t\,dt.
\end{equation}
Integrating by parts gives
\begin{equation}\label{eqn:parts-1}
\int_b^T \frac {e^{ht}}t \,dt
= \frac{e^{ht}}{ht}\Big|_b^T + \int_b^T \frac{e^{ht}}{ht^2}\,dt
\geq \frac{e^{hT}}{hT} - \frac{e^{hb}}{hb},
\end{equation}
and similarly
\[
\int_b^{T+\epsilon} \frac {e^{ht}}t \,dt
= \frac{e^{ht}}{ht}\Big|_b^{T+\epsilon} + \int_b^{T+\epsilon} \frac{e^{ht}}{ht^2}\,dt
\leq \frac{e^{h(T+\epsilon)}}{h(T+\epsilon)} 
+ \frac{1}{hb} \int_b^{T+\epsilon} \frac{e^{ht}}{t} \,dt,
\]
which yields
\[
\Big( 1 - \frac 1{hb}\Big) \int_b^{T+\epsilon} \frac{e^{ht}}t\,dt
\leq e^{h\epsilon} \frac{e^{hT}}{hT}.
\]
Choosing $b$ sufficiently large that $1-\frac 1{hb} \geq e^{-h\epsilon}$, we can combine this with \eqref{eqn:int-bounds} and \eqref{eqn:parts-1} to obtain
\[
e^{-2Q\epsilon} \Big(\frac{e^{hT}}{hT} - \frac{e^{hb}}{hb}\Big)
\leq \  \#(P(T) \setminus P(b))\leq e^{2Q\epsilon} e^{2h\epsilon}\frac{e^{hT}}{hT}.
\]
Sending $T\to\infty$ gives
\[
\#P(T) \sim \#(P(T) \setminus P(b)) \sim e^{\pm 2(Q+h)\epsilon} \frac{e^{hT}}{hT}.
\]
Since $\epsilon>0$ can be arbitrarily small, this implies \eqref{eqn:margulis} and completes the proof.

\appendix
\section{Proofs of basic geometric results}\label{sec:geom-pf}


\begin{proof}[Proof of Lemma \ref{lem:hom-lift}]
Let $\{\hc_s \colon \R/\ell_s\Z \to M\}_{s\in[0,1]}$ be a homotopy between $\hc_0$ and $\hc_1$. Lifting to a homotopy $\{\tc_s \colon \R\to X\}_{s\in [0,1]}$ we observe that for each $s\in [0,1]$ there is $\gamma_s\in \Gamma$ such that $\tc_s(t+\ell_s) = \gamma_s \tc_s(t)$ for all $t\in\R$. Moreover, $s\mapsto \gamma_s$ is continuous, hence constant, so $\gamma_1 = \gamma_0$, which proves the lemma.
\end{proof}


\begin{proof}[Proof of Lemma \ref{lem:2-axes}]
Since $\gamma$ acts isometrically, we have 
\[
d(\tc_0(0),\tc_1(0)) = d(\gamma^n \tc_0(0), \gamma^n \tc_1(0))
= d(\tc_0(n\ell_0),\tc_1(n\ell_1))
\]
for all $n\in\N$, and the triangle inequality gives
\begin{align*}
n|\ell_1 - \ell_0| &= d(\tc_1(n\ell_1),\tc_1(n\ell_0)) \\
&\leq d(\tc_1(n\ell_1),\tc_0(n\ell_0))
+ d(\tc_0(n\ell_0),\tc_1(n\ell_0)) \\
&\leq d(\tc_0(0),\tc_1(0)) + d(\tc_0(n\ell_0),\tc_1(n\ell_0)).
\end{align*}
Lemma \ref{lem:endfix} gives $\tc_1(\infty) = \tc_0(\infty)$, so $\sup_{n} d(\tc_0(n\ell_0),\tc_1(n\ell_0)) < \infty$, and we conclude that $\sup_{n} n|\ell_1 - \ell_0| < \infty$, which implies $\ell_1 = \ell_0$. Applying this result to the geodesic shows that the common value is $|\gamma|$.

To prove that $\hc_0$ and $\hc_1$ lie in the same free homotopy class, let $s\mapsto \tc_s(0)$ be any path from $\tc_0(0)$ to $\tc_1(0)$, and define $\tc_s(|\gamma|) := \gamma \tc_s(0)$. This defines $\tc_s(t)$ as a continuous function of $(s,t)$ on the boundary of $[0,1]\times [0,|\gamma|]$. Since $X$ is simply connected this extends to a continuous map on all of $[0,1]\times [0,|\gamma|]$, and then to $[0,1]\times \R$ by defining $\tc_s(t\pm|\gamma|) := \gamma^{\pm 1} \tc_s(t)$; this gives the desired homotopy.
\end{proof}


\begin{proof}[Proof of Lemma \ref{lem:cpt}]
By continuity of $H$, it suffices to show that $H^{-1}(\Pa\times\Fu\times\{0\})$ is bounded.
Let $H_0$ be the Hopf map for the background metric $g_0$.  It is a homeomorphism, so $H_0^{-1}(\Pa\times\Fu\times\{0\})$ is compact, hence bounded. Using the Morse lemma (see in particular \cite[Lemma 2.4]{CKW}), $H_0^{-1} \circ H \colon SX\to SX$ has the property that there is $R>0$ such that $d(H_0^{-1}\circ H(v),v) \leq R$ for all $v\in SX$, and $H^{-1}(\Pa\times\Fu\times \{0\})$ is contained in the $R$-ball around the bounded set $H_0^{-1}(\Pa\times\Fu\times\{0\})$.
\end{proof}


\begin{proof}[Proof of Lemma \ref{lem:inj0}]
Assume that this is not the case. 
Then there exist sequences $\theta_n \searrow 0$ and $v_n,w_n \in H^{-1}(\Pa_{\theta_n}\times\Fu_{\theta_n}\times \{0\})$ such that $d(\pi v_n, \pi w_n) \geq \frac\epsilon2$. 
By Lemma \ref{lem:cpt}, $H^{-1}(\Pa_{\theta_n}\times\Fu_{\theta_n}\times \{0\})$ is compact, so there is a subsequence $n_k\to\infty$ such that $v_{n_k} \to v$ and $w_{n_k} \to w$, with $d(v,w) \geq \frac\epsilon2$ and $v,w\in H^{-1}(\Pa_{\theta_n} \cap \Fu_{\theta_n} \times \{0\})$ for every $n$. This implies that $v\neq w$ and that $E(v) = E(w) = (v_0^-,v_0^+)$, contradicting the assumption that $v_0 \in \mathcal{E}$.
\end{proof}

Lemma \ref{lem:eps'} is an immediate consequence of the following general result, which we prove using Corollary \ref{cor:bus-cts}.

\begin{lemma}\label{lem:unif-cts}
Let $M$ be a closed Riemannian manifold without conjugate points that satisfies the uniform visibility condition and admits a background metric of negative curvature, and let $X$ be its universal cover. Fix $p\in X$. Suppose $A \subset S_p X$ is closed and $B\subset X$ is such that $A^+ := \{v^+ : v\in A\}$ and $B^\infty := \{\lim_n q_n \in \ideal : q_n \in B\}$ are disjoint subsets of $\ideal$.
Then the family of functions $A\to \R$ indexed by $B$ and given by $v \mapsto b_v(q)$ are equicontinuous: for every $\epsilon>0$ there exists $\delta>0$ such that if $\measuredangle_p(v,w) < \delta$, then $|b_v(q) - b_w(q)| < \epsilon$ for every $q\in B$.
\end{lemma}
\begin{proof}
We start by defining $b_x(q,p)$ and $\beta_p(x,y)$ when $x,y\in X$, not just when $x,y\in \ideal$.
Given $x,q,p\in X$, let
\begin{equation}\label{eqn:bx}
b_x(q,p) := d(q,x) - d(p,x).
\end{equation}
This represents how much longer it takes to get to $x$ if you start at $q$, compared to starting at $p$. When $x\in \ideal$, recall that $b_x(q,p)$ is the Busemann function defined by Definition \ref{def:busemann}, and if a sequence $x_n \in X$ converges to $x\in \ideal$, then $b_{x_n}(q,p) \to b_x(q,p)$; see \cite[Corollary 2.18]{CKW}.
Thus by Corollary \ref{cor:bus-cts}, $(x,q,p) \mapsto b_x(q,p)$ is continuous on $\bar{X} \times X \times X$.

Now given $x,y,p\in X$, let
\begin{equation}\label{eqn:beta}
\beta_p(x,y) = d(x,p) + d(y,p) - d(x,y).
\end{equation}
This represents the extra distance it takes to travel via $p$ when going from $x$ to $y$. Given any point $q$ on the geodesic from $x$ to $y$, we have $d(x,y) = d(x,q) + d(q,y)$ and thus
\begin{equation}\label{eqn:beta2}
\begin{aligned}
\beta_p(x,y) &= d(x,p) + d(y,p) - d(x,q) - d(y,q) \\
&= -(b_x(q,p) + b_y(q,p)).
\end{aligned}
\end{equation}
Comparing this to \eqref{eqn:beta-p}, we see that $(p,x,y) \mapsto \beta_p(x,y)$ is a continuous function on $X\times (\bar{X}\times\bar{X} \setminus \{(\xi,\xi) : \xi\in \ideal\})$. (Recall that if $x\in \ideal$ then the value of $\beta_p(x,x)$ would be infinite.) Moreover, given $p,q\in X$ and $\xi\in \ideal$, we have
\[
\beta_p(q,\xi) = -b_\xi(q,p) - b_q(q,p) = -b_\xi(q,p) + d(q,p).
\]
Now in the setting of Lemma \ref{lem:unif-cts}, we have fixed $p\in X$ and see that for each $q\in B$ and $v\in A$ we have
\begin{equation}\label{eqn:b-beta}
b_v(q) = b_{v^+}(q,p) = d(q,p) - \beta_p(q,v^+).
\end{equation}
Let $\bar{B}$ be the set of all limit points of $B$ in $\bar{X}$, and observe that $A^+ \times \bar{B}$ is a compact set on which $\beta_p$ is continuous; this is where we use the assumption that $A^+ \cap B^\infty = \emptyset$. Thus $\beta_p$ is uniformly continuous on this set, and for every $\epsilon>0$ there is $\delta>0$ such that if $\measuredangle_p(v,w) < \delta$, then $|\beta_p(q,v^+) - \beta_p(q,w^+)| < \delta$, in which case \eqref{eqn:b-beta} gives $|b_v(q) - b_w(q)| < \epsilon$, which completes the proof of Lemma \ref{lem:unif-cts}.
\end{proof}

\bibliographystyle{amsalpha}
\bibliography{counting}

\end{document}